\documentclass[11pt,reqno,a4paper, english]{amsart}
\usepackage[top=3.5cm,bottom=3cm,left=2.5cm,right=2.5cm]{geometry}
\usepackage[mathscr]{euscript}

\usepackage[utf8]{inputenc}
\usepackage[main=english, french, german, latin]{babel}
\usepackage[T1]{fontenc}

\usepackage{amsmath, mathtools, amsthm, amsfonts, amssymb,xcolor}
\usepackage{microtype}
\usepackage{hyperref}
\hypersetup{colorlinks={true},linkcolor={blue},citecolor=blue}
\usepackage{comment}
\usepackage{mathrsfs}

\usepackage{enumitem}

\numberwithin{equation}{section}

\theoremstyle{plain}
\newtheorem{theorem}{Theorem}[section]
\newtheorem{proposition}[theorem]{Proposition}

\newtheorem{lemma}[theorem]{Lemma}

\theoremstyle{definition}

\newtheorem{remark}[theorem]{Remark}

\DeclareMathOperator{\supp}{supp}

\DeclareMathOperator{\e}{e}

\newcommand{\N}{\mathbb{N}}
\newcommand{\Z}{\mathbb{Z}}

\newcommand{\R}{\mathbb{R}}
\newcommand{\C}{\mathbb{C}}
\newcommand{\T}{\mathbb{T}}

\newcommand{\abs}[1]{\left\vert #1 \right\vert}
\newcommand{\parent}[1]{\mathopen{}\Big(#1\Big)}
\newcommand{\set}[1]{\mathopen{}\big\{#1\mathclose{}\big\}}

\renewcommand{\epsilon}{\varepsilon}

\makeatletter
\@namedef{subjclassname@2020}{\textup{2020} Mathematics Subject Classification}
\makeatother

\begin{document}

\title[Lack of smoothing for NLS on $\mathbb{S}^2$]{The Second Picard iteration of NLS on the $2d$ sphere does not regularize Gaussian random initial data}

\author[N. Burq]{Nicolas Burq}
\address{Universit\'e Paris-Saclay, Laboratoire de Math\'ematique d’Orsay, UMR CNRS 8628, Orsay,
France, and Institut universitaire de France }
\email{nicolas.burq@universite-paris-saclay.fr}

\author[N. Camps]{Nicolas Camps}
\address{Univ Rennes, IRMAR - UMR CNRS 6625, F-35000 Rennes, France}
\email{nicolas.camps@univ-rennes.fr}

\author[M. Latocca]{Mickaël Latocca}
\address{Université d'Evry, Laboratoire de Math\'ematiques et de Mod\'elisation d'Evry, UMR CNRS 8071, Evry, France }
\email{mickael.latocca@univ-evry.fr}

\author[C. Sun]{Chenmin Sun}
\address{CNRS, Universit\'e Paris-Est Cr\'eteil,  Laboratoire d'Analyse et de Math\'ematiques Appliqu\'ees, UMR CNRS 8050, Cr\'eteil, France}
\email{chenmin.sun@cnrs.fr}

\author[N. Tzvetkov]{Nikolay Tzvetkov}
\address{Ecole Normale Sup\'erieure de Lyon, Unit\'e de Math\'ematiques Pures et Appliqu\'es,  UMR CNRS 5669, Lyon, France}
\email{nikolay.tzvetkov@ens-lyon.fr}

\date{\today}


\begin{abstract}
We consider the Wick ordered cubic Schrödinger equation (NLS) posed on the two-dimensional sphere, with initial data distributed according to a Gaussian measure. We show that the second Picard iteration does not improve the regularity of the initial data in the scale of the classical Sobolev spaces. This is in sharp contrast with the Wick ordered NLS on the two-dimensional tori, a model for which we know from the work of Bourgain that the second Picard iteration gains one half derivative. Our proof relies on identifying a singular part of the nonlinearity. We show that this singular part is responsible for a concentration phenomenon on a large circle (i.e. a stable closed geodesic), which prevents any regularization in the second Picard iteration.
\end{abstract}

\ \vskip -1cm  \hrule \vskip 1cm \vspace{-8pt}
 \maketitle 
{ \textwidth=4cm \hrule}

\maketitle

%%%%%%%%%%%%%%%%%%%% END OF TITLE PAGE %%%%%%%%%%%%%%%%%%%%%%%%
\setcounter{tocdepth}{1}
\tableofcontents

\section{Introduction}

% \subsection{The nonlinear Schrödinger equation on a compact manifold}

\subsection{Context}

\medskip

The present work is motivated by the study of the influence of the  background geometry on the low regularity well-posedness theory for nonlinear Schrödinger equations (and more generally for dispersive partial differential equations). 
Due to infinite propagation speed, even the short time nonlinear evolution is sensitive to the geometry and (at least for certain geometries based on the model case of the $2d$ sphere) high frequencies instabilities occur (see e.g. \cite{BGT2002Insta,BGT2010}). 
%at regularities above the Euclidean Sobolev regularity threshold given by the scaling invariance. 
%
%
%As explained in the subsequent paragraphs, we known from a series of works from Burq--Gérard--Tzvetkov that the deterministic well-posedness theory is indeed very different on the $2d$ sphere, as 
We prove in this work that these instabilities persist for randomized initial data, which exposes a fundamental obstruction for extending to the $2d$ sphere the probabilistic well-posedness theory for NLS  that was developed in the case of Euclidean geometries over the past thirty years.

\medskip

In the case of  integrable partial differential equations, the best results concerning low regularity well-posedness are exploiting fine properties of the Lax pair structures yielding thus results going beyond the scope of applicability of more traditional PDE techniques. Thomas Kappeler was a pioneer in the use of integrability techniques in the context of low regularity well-posedness of dispersive partial differential equations, see \cite{KT,GK,GKP}. Integrability methods will not be used in the present paper but we believe that using random data techniques in the context of integrable partial differential equations is an interesting line of research. 

\subsubsection{The nonlinear Schrödinger equation on compact surfaces}

\medskip

Given $(M,g)$ a compact Riemannian surface without boundary, the cubic Schrödinger equation (NLS) posed on $M$ reads
\begin{equation}
\tag{NLS}
\label{eq:nls}
i\partial_{t} u + \Delta_{g} u = \lambda|u|^{2}u\,,\quad (t,x)\in\R\times M\,,
\end{equation}
where $\lambda\in\R$ dictates the attractive ($\lambda<0$) or repulsive ($\lambda>0$) nature of the nonlinear interaction. When $s>1$ the Sobolev embedding and a fixed point argument easily solves the Cauchy problem associated to \eqref{eq:nls}: for every bounded set $B\subset H^{s}(M)$ there exist $T_{B}>0$ and a unique solution map
\[
\Phi: u_{0}\in B\ \longmapsto\ (t\mapsto\Phi^{t}(u_{0}))\in C([-T_{B},T_{B}]; H^{s}(M))\,. 
\]
For $u_{0}\in B$, the mapping $t\in[-T_{B},T_{B}]\mapsto \Phi^{t}(u_{0})$ solves \eqref{eq:nls} in the integrated form, through the Duhamel formula, with initial data $u_{0}$:
\begin{equation}
\label{eq:duh}
\Phi^{t}(u_{0})= \e^{it\Delta}u_{0} - i\lambda\int_{0}^{t}\e^{i(t-t')\Delta}|\Phi^{t'}(u_{0})|^{2}\Phi^{t'}(u_{0})
	\mathrm{d}t'\,.
\end{equation}
Moreover, $\Phi$ is uniformly continuous (it is actually analytic). When solving the fixed point problem \eqref{eq:duh} by a Picard iteration scheme, we write 
\begin{equation}
\label{eq:picard}
\Phi^{t}(u_{0}) 
	= \e^{it\Delta}u_{0} 
	- i\lambda\int_{0}^{t}\e^{i(t-t')\Delta}|\e^{it'\Delta}u_{0}|^{2}\e^{it'\Delta}u_{0}\, \mathrm{d}t' +\  \textit{higher order expansions.}
\end{equation}
The first nonlinear term is precisely the second Picard iteration.

\medskip
Motivated by the conservation laws associated with the equation, the main question is to determine the largest Sobolev space in which the flow map extends (uniformly) continuously. 
The $L^{p}$- properties of the eigenfunctions of $-\Delta$ play a key role in this program.

\medskip

For general $(M,g)$, the $L^{p}$ estimates on the eigenfunctions come with some derivative loss. These estimates are well-known from the works of H{\"o}rmander \cite{hormander-sp} and Sogge \cite{sogge0}, and are recalled in \eqref{eq:Sogge}. They turn out to be optimal for the $2d$ sphere. Then, motivated by the study of nonlinear waves, Burq--Gérard--Tzvetkov \cite{BGT2004} obtained semiclassical $L_{t}^{2}L_{x}^{\infty}$ Strichartz estimates with $\frac{1}{2}$-derivative loss, and extended the uniform local well-posedness up to $H^{s}(M)$ with $s>\frac{1}{2}$. So far, it is the lowest  common regularity where uniform local well-posedness is known for all $(M,g)$. Besides, in the defocusing case $(\lambda>0)$, the (coercive) conserved energy combined with the local well-posedness in $H^{1}(M)$ leads to global well-posedness in $H^{s}(M)$ for all $s\geq1$. Hani \cite{hani12} extended the range of global well-posedness in $H^{s}(M)$ with $s>\frac{2}{3}$.  

\medskip

In the Euclidean case, when $M=\mathbb{T}^{2}$ is the flat tori Bourgain \cite{Bourgain1993-0,Bourgain1993} proved that the scaling-invariance dictates the well-posedness threshold: the flow map extends uniformly continuously in $H^{s}(\mathbb{T}^{2})$ for all $s>0$ (see also the very interesting recent work \cite{HK} for the global well-posedness in this range). 
This fails when $s<0$ \cite{ChristCollianderTao2003norm,ChristCollianderTao2003Freq}, and the endpoint $s=0$ is a challenging open problem. The Strichartz estimate due to Bourgain--Demeter \cite{bourgain-demeter-15} extends the local well-posedness result for $s>0$ to the case of irrational tori.

\medskip

Outside the Euclidean case, the only fairly well-understood situation is the case of the $2d$ sphere or more generally of a Zoll surface (a surface on which the geodesic flow is periodic, see \cite{besse-zoll}). In theses geometries, the Cauchy problem is uniformly well posed \cite{BGT2005Inv} up to $s>\frac{1}{4}$ and this is optimal \cite{BGT2002Insta,Banica04Sphere}. 

\medskip

Firstly, the uniform well-posedness result above $H^{\frac{1}{4}}(\mathbb{S}^{2})$ follows from a bilinear refinement of the Strichartz estimate due to Burq--Gérard--Tzvetkov \cite{BGT2005Inv}. This bilinear estimate results from the localization of the spectrum of the Laplace operator on a Zoll surface combined with a general bilinear estimate on the spectral projectors (which is true on any surface). 

\medskip

Conversely, instabilities in $H^{s}(\mathbb{S}^{2})$ when $s\leq\frac{1}{4}$ arise from eigenfunctions that concentrate on a stable closed geodesic, such as the equator in the case of the sphere (and more generally on a surface with a closed stable geodesic \cite{thomann-08}). These eigenfunctions are the highest weight spherical harmonics, and  were also used to construct stationary solution in the defocusing case \cite{weinstein-stab}. In the present paper, we use a broader family of spherical harmonics with high weight to evidence instabilities in a probabilistic setting. 
%
%\begin{equation}
%\tag{NLS}
%\begin{cases}
%     \label{eq:NLS0}
%    i\partial_t u + \Delta_g u = |u|^2u\,,\quad (t,x)\in\R\times M\,. \\
%    u|_{t=0} = u_0\,,
%\end{cases}
%\end{equation}
%The dynamics associated with~\eqref{eq:NLS0} 

\subsubsection{Statistical approaches to nonlinear waves.}

In \cite{BurqLebeau14}, Burq and Lebeau considered a natural probability measure on the space of spherical harmonics and proved that almost every random orthonormal basis is uniformly bounded in $L^{p}(\mathbb{S}^{2})$, when $p<+\infty$, contrasting with the $L^{p}$-bounds discussed above. For instance, the $n$-th highest weight spherical harmonics have their $L^{4}$-norm growing like $n^{\frac{1}{8}}$ due to concentration around a large circle. Former works \cite{vander-sphere,shiffman-zelditch,ayache-tzvetkov-08} also go in the direction of improving the $L^{p}$-estimates for generic eigenfunctions. 

\medskip

The natural question is whether this enhanced Sobolev embedding for generic functions on the sphere can lead to probabilistic well-posedness below the deterministic threshold $s=\frac{1}{4}$ or not.
 
 \medskip

When $s$ is small enough such that instabilities are known to occur in $H^{s}$ for some initial data, the probabilistic well-posedness theory initiated by Burq--Tzvetkov \cite{burq-tzvetkov-2008I,burq-tzvetkov-2008II} after the pioneering work of Bourgain \cite{bourgain96-gibbs}, consists in the construction of full-measure subsets of $H^{s}$ (for some natural measures) made of initial data that lead to strong solutions. In addition, the obtained probabilistic solution can be seen as the unique limit in $H^{s}$ of the (deterministic) flow applied to (suitable) smooth approximations of the initial datum.

\medskip

Additionally, this approach is motivated by the construction of global recurrent solutions using the Gibbs invariant measures (supported in $\cap_{\epsilon>0}H^{-\epsilon}(M)$) and, more broadly, the study of the transport of Gaussian measures by nonlinear flows. In the Euclidean setting, or more recently when the target space is a Riemannian manifold \cite{bls-23} the probabilistic approach has led to several significant advances. An overview is much beyond the scope of this introduction. Nonetheless, to motivate our main result, we recall the key mechanism driving the probabilistic method. 

\medskip

In the Euclidean case  the probabilistic decoupling between the Fourier modes of the initial datum leads to nonlinear smoothing effects. Namely, the second Picard iteration gains, say,  $\sigma$-derivatives (for some fixed $\sigma>0$) for almost-every initial data (see \cite{bourgain96-gibbs} and Theorem \ref{thm:tori}). The standard linear-nonlinear decomposition trick consists in solving the fixed-point for the re-centered solution 
\[
v(t):= \Phi^{t}(u_{0}) - \e^{it\Delta}u_{0}\,,\quad v(0)=0\,,
\]
(formally) solution to \eqref{eq:nls} with a stochastic forcing term (which corresponds to the second Picard iteration in \eqref{eq:duh}), and mixed terms depending both on $v$ and on $\e^{it\Delta}u_{0}$. By understanding the mixed terms, one may indeed run a fixed point argument for $v$ in $H^{s+\sigma}$, provided $s+\sigma$ is greater than the deterministic regularity threshold. 

\subsection{Set up and main results of the present work} The current paper is a preliminary step towards a probabilistic well-posedness theory for the cubic nonlinear Schrödinger equation posed on a non-Euclidean compact surface.  In this work, we consider randomized initial data distributed according to a Gaussian measure, and we prove that the second Picard iteration does not gain any regularity. This is in sharp 
contrast with the case of the torus, and precludes the construction of strong solutions by using the Bourgain linear-nonlinear decomposition. 
$\\$

In order to make our statement  precise, we introduce some notations. We consider the  NLS equation on $\mathbb{S}^{2}$ with a Wick ordered nonlinearity:
\begin{equation}
\tag{NLS}
\label{eq:NLS}
i\partial_{t} u - (-\Delta+1)u = :|u|^{2}u:\quad (t,x)\in\R\times\mathbb{S}^{2}\,,
\end{equation}
where $\Delta$ is the Laplace--Beltrami operator on $\mathbb{S}^{2}$ and the renormalized nonlinearity is\footnote{There is another convention of wick cubic power defined formally as $\big(|u(x)|^2-2\mathbb{E}\|u\|_{L^2(\mathbb{S}^2)}^2\big)u(x)$. As $|u(x)|^2-\mathbb{E}\|u\|_{L^2(\mathbb{S}^2)}^2$ is well-defined for almost all the initial data we shall consider later (distributed according to the Gaussian free field) these two conventions are equivalent.   }
\[
:|u|^{2}u: (x)= (|u(x)|^{2}-2\|u\|_{L^{2}(\mathbb{S}^{2})}^{2})u(x)\,,\quad x\in\mathbb{S}^{2}\,.
\]
The norm $L^{2}(\mathbb{S}^{2})$ is associated to the scalar product
\[
\langle f\mid g\rangle = \int_{\mathbb{S}^{2}}f(x)\overline{g}(x)\mathrm{d}\sigma(x)\,,
\]
where we the Lebesgue measure on $\mathbb{S}^{2}$ is normalized such that the sphere has volume~1 (see \eqref{eq:leb}).  
 
 \medskip
 
On the torus $\T^{2}$, the Wick ordering removes the nonlinear interactions that are not regularizing. Indeed, since the mass is preserved $\|u\|_{L^{2}}^{2}u$ is a linear term, which is not regularizing. It is more subtle in the case of the sphere. The reason is that the amplitudes of the eigenfunctions depend in a complicated  way  on the physical space variable $x$, whereas, on flat tori, the plane waves have a constant amplitude equals to 1. This is point is discussed in Remark \ref{rem:wick}. We stress out that if $u$ is a solution to \eqref{eq:NLS} then 
\begin{equation}
\label{eq:gauge}
v(t):=\e^{it-2it\|u\|_{L^{2}}^{2}}u(t)
\end{equation} is a solution to the standard cubic NLS
\begin{equation}
\label{eq:NLS-stand}
i\partial_{t} v+\Delta v =|v|^{2}v\,.
\end{equation}
Hence the regularity properties studied in this article also holds for the solutions of~\eqref{eq:NLS-stand} thanks to the Gauge transform \eqref{eq:gauge}. The second Picard iteration of \eqref{eq:NLS} reads 
\[
\mathcal{I}_{\mathbb{S}^{2}}(t,u_{0}) = -i\int_{0}^{t}\e^{i(t-t')(\Delta-1)}
	\big(
	:|\e^{it'(\Delta-1)}u_{0}|^{2}\e^{it'(\Delta-1)}u_{0}:
	\big)\mathrm{d}t'\,.
\]
Let us turn to the spectral properties of $-\Delta+1$, as the self-adjoint operator with domain $H^{2}(\mathbb{S}^{2})$. It has a discrete spectrum with eigenvalues 
\[
\lambda_{n}^{2}=n^{2}+n+1\,, n\in\N\,,
\] 
so that $\lambda_{n}\sim n $ at infinity. The eigenvalues are degenerate in the sense that they have multiplicity $2n+1$, and the eigenspace is spanned by the spherical harmonics of degree $n$, denoted by $(Y_{n,k})_{|k|\leq n}$. They are the restriction to $\mathbb{S}^{2}$ of the harmonic homogeneous polynomials of degree $n$. We recall some properties in Section \ref{sec:sph-har}, and prove a concentration property of many of them on a geodesic circle.  

\medskip

For $n\in\N$, we denote by $\mathcal{E}_{n}$ the eigenspace 
\[
\mathcal{E}_{n} = \ker(-\Delta+1-\lambda_{n}^{2}\operatorname{Id}) = \mathrm{span}_{\C}\{  \mathbf{b}_{n,k}  \ |\ |k|\leq n\}\,,
\]
where $(\mathbf{b}_{n,k})_{|k|\leq n}$ is a general orthonormal basis of $\mathcal{E}_{n}$. This gives a spectral resolution of $L^{2}(\mathbb{S}^{2})$
\[
L^{2}(\mathbb{S}^{2}) = \bigoplus_{n\geq 0}\mathcal{E}_{n}\,,\quad \mathcal{E}_{n}=\pi_{n}L^{2}(\mathbb{S}^{2}) 
\]
 where $\pi_{n}$ is the orthogonal projector on $\mathcal{E}_{n}$ defined by 
\[
\pi_{n} = \sum_{|k|\leq n}\langle\,\cdot\,|\mathbf{b}_{n,k}\rangle \mathbf{b}_{n,k}\,.
\]
The Sobolev norm $H^{s}(\mathbb{S}^{2})$ is equivalent to 
\[
\Big(\sum_{n\geq0}\lambda_{n}^{2s}\|\pi_{n}f\|_{L^{2}(\mathbb{S}^{2})}^{2}\Big)^{\frac{1}{2}}\,.
\]
Given $N\in\N$, 
\[
P_{\leq N} u= \sum_{n\,:\,\lambda_n \leq N} \pi_{n}u
\]
is the orthogonal projection on the space spanned by the eigenfunctions with eigenvalues $\leq N$.
$\\$

We can now define the Gaussian measures. Fix a probability space $(\Omega,\mathcal{F},\mathbb{P})$, an orthonormal basis $(\mathbf{b}_{n,k})_{n\in\N\,, |k|\leq n}$ made of eigenfunctions of $-\Delta_{\mathbb{S}^{2}}$, and i.i.d. complex standard Gaussian random variables $(g_{n,k})_{n\in\N\,,|k|\leq n}$, namely 
\[
g_{n,k}(\omega) = \frac{1}{\sqrt{2}}(\mathfrak{g}_{n,k}(\omega)+i\mathfrak{h}_{n,k}(\omega))\,,
\]
where $\mathfrak{g}_{n,k}$ and $\mathfrak{h}_{n,k}$ are independent real-valued standard Gaussian random variables on $(\Omega,\mathcal{F},\mathbb{P})$. Given $\alpha\in\R$, the Gaussian measure $\mu_{\alpha}$ is the probability measure induced by the mapping 
\begin{equation}
\label{eq:randomData}
\omega\in\Omega \longmapsto \phi_{\alpha}^{\omega} = \sum_{n\geq 0}\lambda_{n}^{-\alpha}\sum_{|k|\leq n}g_{n,k}(\omega)\mathbf{b}_{n,k}\,.
\end{equation}
The case $\alpha=1$ corresponds to the Gaussian free field, which is used to define the formally invariant Gibbs measure. We detail the properties of functions in the support of $\mu_{\alpha}$ in Section \ref{sec:randomization}. At this stage, it is important to note that the law of Gaussian measures does not depend on the choice of the orthonormal basis $(\mathbf{b}_{n,k})$. Nevertheless, to evidence some instabilities we will work with the particular basis made of spherical harmonics $(Y_{n,k})$, presented in Section \ref{sec:sph-har}. 

\medskip

Note also that the typical regularity of $\phi_{\alpha}^{\omega}$ is $H^{\alpha-1-0}(\mathbb{S}^{2})$ in the sense that for all $\epsilon>0$, 
\[
\mu_{\alpha}(H^{\alpha-1-\epsilon}(\mathbb{S}^{2}))=1\,, \quad \text{but}\quad \mu_{\alpha}(H^{\alpha-1}(\mathbb{S}^{2}))=0\,.
\] 
According to the Weyl's law, this regularity property does not depend on the surface $M$. In the particular case of the sphere, our main result is that the typical regularity of the second Picard iteration is not better than the regularity of $\phi_{\alpha}^{\omega}$. 
\begin{theorem}
\label{th:main}
Fix $t\geq0$ and $\alpha>\frac{1}{2}$. There exist $N_0\in\N$ and $\eta>0$ such that for all $N\geq N_0$\,, 
\begin{equation}
   \label{eq:FirstIterateSphere}
    \eta|t|\ln(N)^{\frac{1}{2}}
    	\leq\|\mathcal{I}_{\mathbb{S}^{2}}(t,P_{\leq N}\phi_{\alpha}^{\omega})
    	\|_{L^{2}(\Omega ; H^{\alpha-1}(\mathbb{S}^{2}))}\,.
\end{equation}
\end{theorem}
We give some comments.
\begin{itemize}
\item The logarithmic divergence of the $H^{\alpha-1}(\mathbb{S}^{2})$-norm should be viewed as a lack of regularization for the second Picard iteration, contrasting with the case of the tori, rational or irrational (see Theorem \ref{thm:tori}). 
\item More precisely, we proved the divergence of quadratic moment of the $H^{\alpha-1}(\mathbb{S}^{2})$-norm of the second Picard iteration of the frequency truncated initial data. This is likely to yield the almost sure divergence of this random process, but we do not have a self-contained elementary proof of this result.
\item It is not clear wether the threshold $\alpha>\frac{1}{2}$ is technical or not. It appears when we prove that the fully-paired interaction are regularizing (so that they do not cancel the divergent term). We refer to \cite{DengNahmodYue1} for a discussion on a notion of probabilistic criticality, which is heuristically related with this threshold.  
\end{itemize}

Theorem \ref{th:main} indicates that the linear evolution is not a good approximation of the solution, even if the initial are randomized. In this light, the structure of the probabilistic solution (if it ever exists, in a suitable sense) associated to $\phi_{\alpha}^\omega$ cannot be 
\begin{equation}
    \label{eq:linearDecomposition}
    u(t)=e^{it(\Delta-1)}\phi_{\alpha}^\omega + \text{smoother remainder}\,.
\end{equation}
In the deterministic setting , the consideration of high-frequency limits in \cite{BGT2010} also indicates that on the $2d$ sphere, when $s\leq\frac{1}{4}$, the linear evolution of an initial data in $H^{s}(\mathbb{S}^{2})$ does not approximate well the solution. 

\medskip

While proving Theorem \ref{th:main}, we isolate a singular resonant interaction between high and low frequencies. This interaction will play a key role in the probabilistic well-posedness theory we develop in the subsequent work~\cite{BCST24}. It is now well established that adapted ansatz in the spirit of paracontrolled calculus allows one to go beyond the linear-nonlinear decomposition of Bourgain \cite{bourgain96-gibbs}. The strategy is to perform an induction on frequency scales, and to absorb the singular $high\times low$ interactions in a linear operator applied to the high-frequency part of the initial data. We refer, for instance, to the introduction of \cite{BDNY2022} for discussion on the modern technics, such as the random averaging operators \cite{DengNahmodYue1,SunTzvetkov2021derivative} and the random tensors \cite{DengNahmodYue2}. The precursors of these methods, for the wave equation, were \cite{gubinelli-koch-oh2021}  and \cite{bringmann2021-ansatz}. Finally, we stress out that in both \cite{bringmann2021-ansatz,Camps-Gassot-Ibrahim} the second Picard iteration is not smoother than the initial data.  This was rigorously proved by Oh \cite{Oh2011Szego} in the case of the Szeg\H{o} equation which is covered by the result in \cite{Camps-Gassot-Ibrahim}.  

\medskip

In a second result, proved in Appendix \ref{sec:appendix},  we propose a self-contained proof of the gain of $\frac{1}{2}$-derivative of the second Picard iteration in the case of irrational tori, for initial data distributed according to the Gaussian free field. We define
\[
\mathcal{I}_{\mathbb{T}_{\beta}^{2}}(t,u_{0}) = \int_{0}^{t}\e^{i(t-t')(\Delta_{\beta}-1)}\Big(:|\e^{it'(\Delta_{\beta}-1)}u_{0}|^{2}\e^{it'(\Delta_{\beta}-1)}u_{0}:\Big)
	\mathrm{d}t'\,,
\]
where, for $\beta>0$, the rescaled Laplacian $-\Delta_{\beta}$ is the rescaled Laplacian:
\[
-\Delta_{\beta}:=\partial_{x_{1}}^{2}+\beta^{2}\partial_{x_{2}}^{2}\,.
\]
On the torus, the Gaussian free field is induced by the random variable 
\[\phi_{1}^\omega (x) = \sum_{n\in\Z^2}\frac{1}{\langle n \rangle}g_{n}(\omega)\e^{in\cdot x}\,,
\]
where $\langle n\rangle =(1+Q(n))^{1/2}$, with for $n=(k,m)\in\Z^{2}$, $Q(n)=k^{2}+\beta^{2}m^{2}$. We have that  $\phi_{1}^{\omega}\in H^{0-}(\T_{\beta}^{2})$ for almost every $\omega$. 
\begin{theorem}[Regularity of the first iteration on tori]\label{thm:tori}
For $\alpha=1$, and for all $\epsilon>0$ there exists $C>0$ such that for all $N\in\N$ and  $t\in\R$,
 \begin{equation}
    \label{eq:FirstIterateTori}
 \|\mathcal{I}_{\T_{\beta}^{2}}(t,P_{\leq N}\phi_{1}^{\omega})\|_{L_{\omega}^{2}(\Omega ; H^{\frac{1}{2}-\epsilon}(\mathbb{T}^{2}))}
 \leq C\,.
\end{equation}
\end{theorem}
This result is implicitly contained  in \cite{bourgain96-gibbs} (see also \cite{fos-21}). For completeness and to put Theorem \ref{th:main} into context we give a short proof. 

\subsection*{Organization of the article} 
The proof of Theorem~\ref{th:main} in the case of the sphere is given in Section~\ref{sec:proof} and mostly relies on a quantitative concentration property, which is proven in Section~\ref{section:equator}.  We also recall or prove preliminary properties  on the spherical harmonics in Section~\ref{sec:sph-har}, and on Gaussian measures on $\mathbb{S}^{2}$ in Section \ref{sec:randomization} together with some preparations on the nonlinearity. Finally, we prove Theorem~\ref{thm:tori} in Section~\ref{sec:appendix}. 
	
\subsection*{Acknowledgments}
	The authors would like to thank the anonymous referee for the careful reading of the manuscript and for the valuable comments and suggestions. This research was supported by the European research Council (ERC) under the European Union’s Horizon 2020 research and innovation programme (Grant agreement 101097172 - GEOEDP). C.S and N.T. were partially supported by the ANR project Smooth ANR-22-CE40-0017.
	N.C. benefited from the support of the Centre Henri Lebesgue ANR-11-LABX-0020-0, the region Pays de la Loire through the project MasCan and the ANR project KEN ANR-22-CE40-0016.

\section{Spherical harmonics}
\label{sec:sph-har}
\subsection{Spherical harmonics}  The polar (colatitudinal) coordinate is $\theta\in (0,\pi)$, the azimuthal (longitudinal) coordinate is $\varphi\in(0,2\pi)$, and $\rho\in\R^+$ is the radial distance, so that $\mathbb{S}^{2}=\{x\in\R^3\mid \abs{\rho}=1\}$. 
We have
\[
(x_1,x_2,x_3)=(\rho\sin(\theta)\cos(\varphi),\rho\sin(\theta)\sin(\varphi),\rho\cos(\theta))
\,.\]
With this coordinates system, the Lebesgue measure is 
\begin{equation}
\label{eq:leb}
\mathrm{d}\sigma = \frac{1}{4\pi}\sin(\theta)\mathrm{d}\theta \mathrm{d}\varphi\,,
\end{equation}
and the Hermitian scalar product is
\begin{equation}
\label{eq:scal}
\langle f\mid g\rangle_{L^2(\mathbb{S}^{2})}=\frac{1}{4\pi}\int_0^{2\pi}\int_0^\pi f(\theta,\varphi)\overline{g(\theta,\varphi)}\sin(\theta)\mathrm{d}\theta \mathrm{d}\varphi\;.
\end{equation}
With this normalization  $\mathbb{S}^{2}$ has volume 1. Note that this coordinate system is singular on the axis $(Ox_3)$. The (shifted) Laplace-Beltrami operator on the sphere reads 
\begin{equation}
\label{eq:deltaS2}
-\Delta_{\mathbb{S}^{2}}+1 = -\frac{1}{\sin^2\theta}\frac{\partial^2}{\partial\varphi^2}-\frac{1}{\sin\theta}\frac{\partial}{\partial\theta}(\sin\theta\frac{\partial}{\partial\theta})+1\;.
\end{equation}
It is a selfadjoint operator with domain $H^{2}(\mathbb{S}^{2})$ and it has discrete spectrum. The eigenvalues are
\[
\lambda_{n}^{2}= n^{2}+n+1\,,\quad n\in\N\,,
\]
with multiplicity $2n+1$. The normalized spherical harmonics, defined as the restrictions to $\mathbb{S}^{2}$ of the harmonic homogeneous polynomials, form a particular orthonormal basis of eigenfunction. For $n\in\N$ and $k\in\{-n,\dots,n\}$, we denote by $Y_{n,k}$ the spherical harmonics of degree $n$ and order $k$.  In spherical coordinates, we have
\begin{equation}
\label{eq:sh}
Y_{n,k}(\theta,\varphi)
	=\e^{ik\varphi}v_{n,k}(\theta),\quad  v_{n,k}(\theta)=c_{n,k}L_{n,k}(\cos(\theta))\,,
\end{equation} 
where $L_{n,k}(\cos(\theta))$ is the associated Legendre function of degree $n$ and order $k$, and $c_{n,k}$ is a normalization constant:
\[
c_{n,k} = \sqrt{(2n+1)\frac{(n-k)!}{(n+k)!}}\,.
\]
\begin{proposition}
\label{prop:harmonics}
Spherical harmonics $(Y_{n,k})_{1\leq n\,,\abs{k}\leq n}$ form an orthonormal basis of $L^2(\mathbb{S}^{2})$ made of eigenfunctions of the Laplace operator on $\mathbb{S}^{2}$. They satisfy  
\[(-\Delta_{\mathbb{S}^{2}}+1)Y_{n,k} = \lambda_n^2Y_{n,k},\quad n\in\N^*,\ k\in\{-n,\dots,n\},\]
with $\lambda_n^{2} =n^{2}+n+1$. \end{proposition}
We refer to \cite{sogge0} and \cite{steinEuclidean}, Chapter IV,  for a detailed analysis of the spherical harmonics.  

\subsection{Concentration of spherical harmonics with high order}
\label{section:equator}
It follows from the expression of the Laplace operator on the sphere~\eqref{eq:deltaS2} that $v_{n,k}$ is solution on $(0,\pi)$ to 
\begin{equation}
\label{eq:v-theta}
-(\sin(\theta)\frac{\mathrm{d}}{\mathrm{d}\theta})^2v_{n,k}(\theta) + (k^2- n(n+1)\sin^2(\theta))v_{n,k}(\theta) = 0.
\end{equation}

The next Proposition claims that in the high frequency regime $n\to\infty$ a large family of  spherical harmonics with high weights (when $|k|$ is close to $n$) concentrate their mass near the equator (corresponding to the region $\theta=\frac{\pi}{2}$). Given $\delta\in(0,1)$, we denote
\begin{equation}
\label{eq:Cdelta}
\mathscr{C}_{\delta} := \Big\{x=(\theta,\varphi)\in\mathbb{S}^{2}\quad |\quad  \delta<|\cos(\theta)|\Big\}\,.
\end{equation}
\begin{proposition}
\label{prop:concentration}
For all $\delta>0$, $n\geq1$ and $k\in\Z$ such that
\begin{equation}
\label{eq:k}
n(n+1)(1-\delta^2)\leq k^{2}\leq n^{2}\,,
\end{equation}
we have
\[
\|\mathbf{1}_{\mathscr{C}_{2\delta}}(x)Y_{n,k}(x)\|_{L_{x}^2(\mathbb{S}^{2})}\lesssim_{\delta}\frac{1}{n}\,.
\]
\end{proposition}

The proof relies on utilizing semiclassical functional calculus in the high-frequency regime to quantitatively exploit the ellipticity of the equation \eqref{eq:v-theta} away from the equator. A different approach, based on ordinary differential equations, can be found in~\cite[Section 3]{FS17}, where families of spherical harmonics with high weight $|k|\sim n$ are used to saturate certain $L^{p}$ spectral cluster bounds.

\begin{proof} 
In order to use  semiclassical functional calculus and prove this proposition, we shall extend $v_{n,k}$ to a function defined on the whole real line $\R$. For this purpose, we make the change of variable
\[
\begin{array}{ccccc}f& \colon &(0,\pi)&\longrightarrow &\R \\&&&&\\& &\theta&\longmapsto & \tanh^{-1}(\cos(\theta))\,.\\\end{array}
\]
Note that $f$ is a $C^\infty$-diffeomorphism and that we have the identity: for all $\theta\in(0,\pi)$,
\[
    \sin^2(\theta) = 1-\cos^2(\theta) = 1- \tanh^2(f(\theta)) = \frac{1}{(\cosh\circ f(\theta))^{2}}\,.
\]
This yields 
\begin{equation}
    \label{eq-diffeo}
    f'(\theta) = -\frac{1}{\sin(\theta)}\,,\quad (f^{-1})'(y) = -\frac{1}{\cosh(y)} \,,
\end{equation}
which in turn implies that 
\begin{equation}
\label{eq-L2-v}
   \int_{-\infty}^\infty \frac{1}{\cosh(y)^{2}}|v_{n,k}\circ f^{-1}(y))|^{2}\mathrm{d}y =
    \int_{0}^\pi|v_{n,k}(\theta)|^2\sin(\theta)\mathrm{d}\theta=\frac{1}{2\pi} \,. 
\end{equation}
We set
\[
    \widetilde v_{n,k}(y) =\frac{(2\pi)^{\frac{1}{2}}}{\cosh(y)}v_{n,k}\circ f^{-1}(y) \,,\quad y\in\R\,,
\]
so that 
\[
\int_{-\infty}^{\infty}|\widetilde{v}_{n,k}(y)|^{2}dy=1\,.
\]
In what follows we may abuse notations and write $\theta=f^{-1}(y)$ for $y\in\R$. We have from the chain rule and from \eqref{eq-diffeo} that 
\begin{equation}
\label{eq:dy}
\frac{\mathrm{d}}{\mathrm{d}y}\widetilde{v}_{n,k}(y) = - \tanh(y)\widetilde{v}_{n,k}(y)-\frac{(2\pi)^{\frac{1}{2}}}{\cosh(y)}\sin(\theta)\frac{\mathrm{d}}{\mathrm{d}\theta}v_{n,k}(\theta)\,.
\end{equation}
We deduce from this that 
\begin{equation}
\label{eq:est-dv}
\|\frac{\mathrm{d}}{\mathrm{d}y}\widetilde{v}_{n,k}\|_{L^{2}(\R)}\lesssim1\,.
\end{equation}
Indeed, 
\[
\|\tanh\widetilde{v}_{n,k}\|_{L^{2}(\R)}\leq \|\widetilde{v}_{n,k}\|_{L^{2}(\R)}\leq1\,,
\]
and 
\begin{align*}
\Big\|\frac{1}{\cosh(y)}\sin(\theta)\frac{\mathrm{d}}{\mathrm{d}\theta}v_{n,k}(\theta)\Big\|_{L^{2}(\R)}^{2} 
	&= \int_{\R} \sin^{2}(\theta)\frac{1}{\cosh(y)}|v_{n,k}'\circ f^{-1}(y)|^{2} |(f^{-1})'(y)|\mathrm{d}y \\
	&=\int_{0}^{\pi}\sin^{3}(\theta)|v_{n,k}'(\theta)|^{2}d\theta \lesssim 1\,.
\end{align*}
Moreover, we deduce from \eqref{eq:dy} that
\begin{multline*}
\frac{\mathrm{d}^{2}}{\mathrm{d}y^{2}}\widetilde{v}_{n,k}(y) 
	= (\tanh^{2}(y)-1)\widetilde{v}_{n,k}(y) - \tanh(y)\frac{\mathrm{d}}{\mathrm{d}y}\widetilde{v}_{n,k}(y)\\  
	+\frac{(2\pi)^{\frac{1}{2}}}{\cosh(y)}\Big(2\cos(\theta)\sin(\theta)\frac{\mathrm{d}}{\mathrm{d}\theta}v_{n,k}(\theta) 
	+\sin^{2}(\theta)\frac{\mathrm{d}^{2}}{\mathrm{d}\theta^{2}}v_{n,k}(\theta)\Big)\,.
\end{multline*}
Plugging in the expression \eqref{eq:dy}, we deduce that 

\[
\frac{\mathrm{d}^{2}}{\mathrm{d}y^{2}}\widetilde{v}_{n,k}(y) 
	= -(1+2\cos(\theta)\frac{\mathrm{d}}{\mathrm{d}y})\widetilde{v}_{n,k}(y) + \frac{(2\pi)^{\frac{1}{2}}}{\cosh(y)}(\sin(\theta)\frac{\mathrm{d}}{\mathrm{d}\theta})^{2}v_{n,k}(\theta)\,.
\]
Therefore, we conclude form \eqref{eq:v-theta} that $\widetilde v_{n,k}\in C^\infty(\R)$ solves
\begin{equation}
    \label{eq:widetildev}
    \frac{\mathrm{d}^2}{\mathrm{d}y^2}\widetilde{v}_{n,k}+(1+2\cos(\theta)\frac{\mathrm{d}}{\mathrm{d}y})\widetilde{v}_{n,k}-(k^2-n(n+1)\sin^{2}(\theta))\widetilde{v}_{n,k}=0\,. 
\end{equation}

We can now introduce the semiclassical parameter in the high-frequency regime $n\to\infty$:
\[
h:= (n(n+1))^{-\frac{1}{2}}\,,\quad h\sim_{n\to\infty} n^{-1}\,.
\]
For the rest of the proof, we denote
\[
\alpha:=|k|h\,. 
\]
This parameter should not be confused with the parameter $\alpha$ in the definition \eqref{eq:randomData} of the initial data, which is not involved in this proof. Multiplying the equation~\eqref{eq:widetildev} by~$h^2$ gives
\begin{equation}
    \label{eq:v-h}
-h^{2}\frac{\mathrm{d}^{2}}{\mathrm{d}y^{2}}\widetilde{v}_{n,k}(y) + (\alpha^{2}-\sin^{2}(\theta(y)))\widetilde{v}_{n,k}(y)  - 2h^{2}\cos(\theta(y))\frac{\mathrm{d}}{\mathrm{d}y}\widetilde{v}_{n,k}(y) - h^{2}=0\,.
\end{equation}
We reformulate \eqref{eq:v-h}
as follows:
\[
\eqref{eq:v-h}\quad \iff\quad  \operatorname{P}_{\alpha}(y,hD_{y})(\widetilde{v}_{n,k})=2h^{2}\cos(\theta(y))\frac{\mathrm{d}}{\mathrm{d}y}\widetilde{v}_{n,k}(y)+h^{2}\,,
\]
where $\operatorname{P}_{\alpha}$ is a differential operator of order 2 with symbol  
\[
   p_{\alpha}(y,\xi) =\xi^2+\alpha^2-\sin^{2}(\theta(y))\,.
\]
We deduce from \eqref{eq:est-dv} that 
\begin{equation}
\label{eq:est-P}
\|\operatorname{P}_{\alpha}(y,hD_{y})\widetilde{v}_{n,k}\|_{L^{2}(\R)}\lesssim h^{2}\,.
\end{equation}
Consider 
\[
\operatorname{Char}(\operatorname{P}_{\alpha}) 
	= \Big\{(y,\xi)\in\R^2\ \mid\ p_{\alpha}(y,\xi)=0\Big\}
	= \Big\{(y,\xi)\in\R^{2}\ \mid\ \xi^2+\alpha^2-\sin^{2}(\theta(y))=0\Big\}\,, 
\]
and set, for $\delta\in(0,1)$, 
\begin{equation}
\label{eq:C}
\widetilde{\mathscr{C}_{\delta}} := \Big\{y\in\R\,\ |\ \delta<|\cos(\theta(y))|\Big\}\,.
\end{equation}
Observe that for all $\delta\in(0,1)$, 
\[
\sqrt{1-\delta^2}\leq \alpha\quad \implies \quad \operatorname{Char}(\operatorname{P}_{\alpha}) 
	 \subset\R\setminus\widetilde{\mathscr{C}_{\delta}}\,. 
\] 
Indeed, 
\[
\begin{cases}
	&\sqrt{1-\delta^2}\leq \alpha\,,\\
	 &\xi^2+\alpha^2-\sin^{2}(\theta) =  0
\end{cases}
	\quad \implies\quad 1-\delta^{2}\leq \sin^{2}(\theta(y))\quad \implies\quad |\cos(\theta(y))| \leq \delta\,.
\]
Hence, in the regime where $\sqrt{1-\delta^{2}}\leq \alpha $,  $\operatorname{Char}(\operatorname{P}_{\alpha})$ concentrates near the equator as $\delta$ goes to 0. By the use of semiclassical functional calculus, we will deduce from this that $\widetilde v_{n,k}$ also concentrates its mass near the equator in this regime. 
 
\begin{lemma}
\label{lemma-ellipticity}
Let $\delta\in(0,1)$, and suppose that $\sqrt{1-\delta^{2}}\leq\alpha$. Then, for all $\chi\in C_{c}^\infty(\R)$ such that
\[
\chi\equiv1\ on\ \widetilde{\mathscr{C}}_{2\delta}\,,\quad \supp\chi\subset \widetilde{\mathscr{C}}_{\frac{3}{2}\delta}\,,
\] 

where $\widetilde{\mathscr{C}}_{\delta}$ is defined in \eqref{eq:C}, there exists a bounded operator $\operatorname{Q}_{\alpha,h}$ such that for all $h\in(0,1)$, 
\[
\|\operatorname{Q}_{\alpha,h}\|_{L^{2}(\R)\to L^{2}(\R)}\lesssim_{\delta}1\,,
\]
and 
\begin{equation}
    \|\operatorname{Q}_{\alpha,h}\circ \operatorname{P}_{\alpha}(y,hD_{y})-\chi(y)\|_{L^2(\R)\rightarrow L^2(\R)}\lesssim_{\delta} h\,. 
\end{equation}
\end{lemma}

\begin{proof}
Fix $\delta\in(0,1)$ and suppose that $\sqrt{1-\delta^{2}}\leq\alpha$. The key point is that $\operatorname{P}_{\alpha}(y,hD_{y})$ is elliptic in the region $\widetilde{\mathscr{C}}_{2\delta}$ and the proof follows from the standard parametrix construction of elliptic operators. Since we only need to invert $\operatorname{P}_{\alpha}(y,hD_{y})$ in the elliptic region up to order 1 in $h$, we propose a self-contained proof. 

\medskip

We introduce the symbol
\[
    q(y,\xi;\alpha):=\frac{\chi(y)}{p_{\alpha}(y,\xi)}\,.
\]
Recall that for all $y\in\supp(\chi)$, $\xi\in\R$ and $\sqrt{1-\delta^2}\leq\alpha$ we have
\begin{equation}
\begin{split}
\label{eq:p0}
p_{\alpha}(y,\xi)
	&=\xi^2+\alpha^2-\sin^{2}(\theta(y))\\
	&\geq 1 - \delta^{2} -  \sin^{2}(\theta(y)) \\ 
	&= \cos^{2}(\theta(y))-\delta^{2}\\
	&\geq \delta^{2}\,.
\end{split}
\end{equation}
The semiclassical pseudodifferential operator $\operatorname{Q}_{\alpha}(y,hD_{y})$ associated with $q$ is defined via its Schwartz kernel 
\[
    \operatorname{Q}_{\alpha}(y,hD_{y})[f](y)=\int_\R K(y,z)f(z)\mathrm{d}z,\quad K(y,z) =\frac{1}{2\pi h}\int_\R\e^{\frac{i(y-z)\xi}{h}}
    q(y,\xi;\alpha)\mathrm{d}\xi\,.
\]

Observe that   
\begin{equation}
\label{eq:phase}
(y-z)e^{\frac{i(y-z)\xi}{h}}=-ih\partial_{\xi}(e^{\frac{i(y-z)\xi}{h}})\;.
\end{equation}
Hence, integrating by parts twice yields that 
\[
(y-z)^2K(y,z)
	=-\frac{h}{2\pi}\int_\R \e^{\frac{i(y-z)\xi}{h}} \partial_\xi^2 q(y,\xi;\alpha)\mathrm{d}\xi\,,
\]
It follows from~\eqref{eq:p0} that for all $k\in\N$, $\alpha,\delta$ such that $\sqrt{1-\delta^2}<\alpha\leq1$ and for all $\xi\in\R$,
\begin{equation}
\label{eq:q-bound}
\sup_{y\in\R}\ |\partial_\xi^k q(y,\xi,\alpha)|
	\lesssim_{k} \delta^{-2k}\langle\xi\rangle^{-2k}.
\end{equation}
Hence,
\[
\sup_{(y,z)\in\R^{2}}(y-z)^{2}|K(y,z)|
	\lesssim h\delta^{-4}
	\parent{\int_\R \langle\xi\rangle^{-4}\mathrm{d}\xi}\lesssim h\delta^{-4}\,.
\]
Moreover, we have the trivial bound
\[
\sup_{(y,z)\in\R^{2}} |K(y,z)|
	\leq h^{-1}\delta^{-2}\parent{\int_\R\langle\xi \rangle^{-2}\mathrm{d}\xi}
	\lesssim h^{-1}\delta^{-4}\,.
\]
Therefore,
\begin{equation*}
\begin{split}
\sup_{z\in\R} \int_\R|K(y,z)|\mathrm{d}y
	&\lesssim \delta^{-4} \sup_{z\in\R}\int_\R\min\Big(\frac{h}{(y-z)^2},\frac{1}{h}\Big)\mathrm{d}y\\
	&\lesssim\delta^{-4}\Big(\int_{|s|\leq h}\frac{1}{h}\mathrm{d}s + \int_{|s|>h}\frac{h}{s^{2}}\mathrm{d}s\Big) \\
	&\lesssim\delta^{-4}\,.
\end{split}
\end{equation*}
Since the computations are the same when changing the role of $y$ and $z$, we prove that when $\sqrt{1-\delta^{2}}\leq\alpha$ then 
\[
\sup_{y\in\R}\int_\R|K(y,z)|\mathrm{d}z+\sup_{z\in\R}\int_\R|K(y,z)|\mathrm{d}y\lesssim \delta^{-4}\,,
\]
uniformly in $h\in(0,1)$. It follows from Schur's test that $\operatorname{Q}_{\alpha}(y,hD_{y})$ is bounded on $L^2(\R)$, uniformly in $h\in(0,1)$, with 
\begin{equation}
\label{eq:L2-bound}
\|\operatorname{Q}_{\alpha}(y,hD_{y})\|_{L^2(\R)\to L^2(\R)}\lesssim \delta^{-4}\,.
\end{equation}
The Schwartz kernel of the operator
\[
\operatorname{R}_{\alpha}(y,hD_{y})\coloneqq \operatorname{Q}_{\alpha}(y,hD_{y})\circ \operatorname{P}_{\alpha}(y,hD_{y})-\chi(y)
\]
is given by 
\[
  \mathcal{K}(y,z)=(\phi(z)-\phi(y))\chi(y)\int_{\R}\e^{\frac{i(y-z)\xi}{h}}\frac{1}{p_{\alpha}(y,\xi)}\frac{\mathrm{d}\xi}{2\pi h} \,.
\]
where 
\[
\phi(y) := \sin^{2}(\theta(y))\,.
\]
We write
\[
\varphi(y,z):=\int_0^1\phi'(tz+(1-t)y)\mathrm{d}t\,,
\]
so that 
\[
\phi(y)-\phi(z) = (y-z)\varphi(y,z)\,.
\]
It follows from \eqref{eq:phase} and integration by parts that
\[
    \mathcal{K}(y,z)=-i\varphi(y,z)\chi(y)\int_{\R}\partial_{\xi}\big(\frac{1}{p_{\alpha}(y,\xi)}\big)
    \e^{\frac{i(y-z)\xi}{h}}\frac{\mathrm{d}\xi}{2\pi}\,. 
\]
We proceed as for the estimate on the kernel of $\operatorname{Q}_{\alpha}$ and and we integrate by parts twice (using \eqref{eq:phase}) to get
\[ 
    |\mathcal{K}(y,z)|\lesssim\delta^{-6}\min(1, \frac{h^2}{|y-z|^2})\,. 
\]
We deduce
\[
    \sup_{z\in\R}\int_{\R}|\mathcal{K}(y,z)|\mathrm{d}y+\sup_{y\in\R}\int_{\R}|\mathcal{K}(y,z)|\mathrm{d}z
    \lesssim\delta^{-6} h\,
\]
and we conclude from the Schur's test that 
\[
\|\mathrm{R}_{\alpha}(y,hD_{y})\|_{L^2\rightarrow L^2}\lesssim\delta^{-6} h
\,.
\]
This concludes the proof of Lemma~\ref{lemma-ellipticity} with an implicit constant $\lesssim\delta^{-6}$.
\end{proof}

We now complete the proof of Proposition~\ref{prop:concentration} as follows. We have from Lemma~\ref{lemma-ellipticity}, from~\eqref{eq:L2-bound} and from~\eqref{eq:est-P} that 
\[
\|\chi\widetilde{v}_{n,k}\|_{L^2(\R)}=\|\mathrm{Q}_{\alpha}(y,hD_{y})\mathrm{P}_{\alpha}(y,hD_{y})\widetilde{v}_{n,k}-\mathrm{R}_{\alpha}(y,hD_{y})\widetilde{v}_{n,k}\|_{L^2(\R)}\lesssim_{\delta} h\,.
\]
Hence, according to \eqref{eq-L2-v},
\[
    \|Y_{n,k}\|_{L^2(\mathscr{C}_{2\delta})} =
    \|\widetilde{v}_{n,k}\|_{L^2(\widetilde{\mathscr{C}_{2\delta}})}\leq
    \|\chi\widetilde{v}_{n,k}\|_{L^2(\R)}\leq C_\delta h\,,
\]
which completes the proof of Proposition~\ref{prop:concentration}.
\end{proof}

\section{Preparations}\label{sec:randomization}
\subsection{Gaussian measures}
\label{sec:gm}
Recall that the Gaussian measure $\mu_{\alpha}$ of parameter $\alpha\in\R$ is defined in \eqref{eq:randomData}. To present some properties of the random functions $\phi_{\alpha}(\omega)$ in the support of $\mu_{\alpha}$, we need to first recall the local Weyl law for $-\Delta_{\mathbb{S}^{2}}$:
\begin{lemma}[Local Weyl's law]
\label{lem:weyl} For every $n\in\N$, every orthonormal basis $(\mathbf{b}_{n,k})$ of $\mathcal{E}_{n}$ and every $x\in\mathbb{S}^{2}$, we have 
\begin{equation}
\frac{1}{2n+1}\sum_{|k|\leq n}|\mathbf{b}_{n,k}(x)|^{2} = 1\,.
\label{eq:weyl}
\end{equation}
\end{lemma}
We detail the very concise proof of this Lemma. 
\begin{proof} Given $n\in\N$ and $(\mathbf{b}_{n,k})_{|k|\leq n}$ an orthonormal basis of $\mathcal{E}_{n}$, the integral kernel of the orthogonal projector $\pi_{n}$ onto $\mathcal{E}_{n}$ is 
\[
K_{n}(x,y) = \sum_{|k|\leq n}\mathbf{b}_{n,k}(x)\overline{\mathbf{b}_{n,k}(y)}\,, \quad (x,y)\in(\mathbb{S}^{2})^{2}\,.
\]
It follows from the rotational invariance of the operator $-\Delta_{\mathbb{S}^{2}}$ that for all $R\in SO_{3}$,
\[
\pi_{n}\circ R = R\circ \pi_{n}\,,
\] 
which in turns implies that for every $(x,y)\in(\mathbb{S}^{2})^{2}$ and $R\in SO_{3}$ 
\[
K_{n}(Rx,y) = K_{n}(x,R^{-1}y)\,,
\]
In particular, for all $x\in\mathbb{S}^{2}$
\[
K_{n}(Rx,Rx) = K_{n}(x,x)\,.
\]
Since $SO_{3}$ acts transitively on $\mathbb{S}^{2}$, we deduce that
\[
x\in\mathbb{S}^{2} \mapsto K_{n}(x,x) = \sum_{|k|\leq n}|\mathbf{b}_{n,k}(x)|^{2}
\]
is a constant function. Integrating over $x\in\mathbb{S}^{2}$ and using the assumption that the spherical harmonics are normalized, we conclude the proof of the identity \eqref{eq:weyl}. 
\end{proof}

Let us now reorganize the terms in the series \eqref{eq:randomData}, which defines the Gaussian measure $\mu_{\alpha}$. Given an orthonormal basis $(\mathbf{b}_{n,k})$, we group the eigenfunctions in clusters of same eigenvalue, with multiplicity $2n+1$, and write 
\[
\phi_{\alpha}^{\omega} = \sum_{n\geq0}\lambda_{n}^{-\alpha}\sum_{|k|\leq n}g_{n,k}(\omega)\mathbf{b}_{n,k} \\
	=\sum_{n\geq0}\widetilde{\lambda}_{n}^{-(\alpha-\frac{1}{2})} e_{n}^{\omega}\,,
\]
where, for $n\in\N$,
\begin{equation}
\label{eq:def_en}
e_n^{\omega} =\frac{1}{\sqrt{2n+1}} \sum_{|k|\leq n}g_{n,k}(\omega) \mathbf{b}_{n,k}\,.
\end{equation}
and $\widetilde{\lambda}_{n} = \lambda_{n} (\frac{2n+1}{\lambda_{n}})^{\frac{1}{2}}.$ Since $\lambda_{n}$ and $\widetilde{\lambda}_{n}$ have the same asymptotic up to a factor $2$, we abuse notations and we keep writing 
\[
\lambda_{n}\approx \widetilde{\lambda}_{n}\,.
\]
For every $n\geq0$, $e^{\omega}_n$ can be seen as a Gaussian vector on $\mathcal{E}_n$, with 
\[
\mathbb{E}(e^{\omega}_n)=0\,,\quad \operatorname{Cov}(e^{\omega}_n) = \frac{1}{2n+1}\operatorname{Id}_{\mathcal{E}_n}\,,
\]  
	$\\$
where $\mathrm{Cov}(e_n^{\omega})$ is the covariance matrix of the Gaussian vector $e_n^{\omega}$. We recall in the next Lemma that the law of a complex standard Gaussian vector is invariant under the action of the unitary group: 

\begin{lemma}For all $X\sim \mathcal{N}_{\C}(0,\frac{1}{2n+1}\mathrm{Id}_{\mathcal{E}_{n}})$ and $A\in U(\mathcal{E}_{n})$, the unitary group of $\mathcal{E}_{n}$, we have
\[
 \mathscr{L}(AX)=\mathscr{L}(X)\,.
 \]
\end{lemma}
\begin{proof}
For $Y$ a random variable we denote by $\varphi_{Y}$  its characteristic function. For all $\xi\in\mathcal{E}_{n}$,
\[
\varphi_{AX}(\xi) = \mathbb{E}[\e^{i\langle\xi|AX\rangle}] = \mathbb{E}[\e^{i\langle A^{\ast}\xi|X\rangle}] = \varphi_{X}(A^{\ast}\xi) = \e^{-\frac{1}{2n+1}\|A^{\ast}\xi\|_{L^{2}}^{2}} = \e^{-\frac{1}{2n+1}\|\xi\|_{L^{2}}^{2}} = \varphi_{X}(\xi)\,.
\]
This proves the Lemma.
\end{proof}
A consequence of the above Lemma is that the law of $e_{n}^{\omega}$ --- and therefore the law $\mu_{\alpha}$ --- does not depend on the choice of the orthonormal eigenbasis of $\mathcal{E}_{n}$. Alternatively, we could fix $x$ and view $e_{n}^{\omega}(x)$ as a complex standard Gaussian random variable. Indeed, $\mathbb{E}[e_{n}^{\omega}(x)]=0$ and, according to  \eqref{eq:weyl},
\begin{equation}
\label{eq:ex}
\operatorname{Var}(e_{n}^{\omega}(x))=\frac{1}{2n+1}\sum_{|k|\leq n}|\mathbf{b}_{n,k}(x)|^{2}=1\,.
\end{equation}
In particular, the law of $e_{n}^{\omega}(x)$ does not depend on the point $x$ on the sphere.
\medskip

Let us now recall the elementary property of complex standard Gaussian variables: for all $n,k$, we have
\[
\mathbb{E}[g_{n,k}]=\mathbb{E}[g_{n,k}^2]=\mathbb{E}[g_{n,k}|g_{n,k}|^{2}] = 0\,.
\]
In particular, we deduce from the above identities and from the definition 
 \[
e_{n}^{\omega}(x)
	=(2n+1)^{-\frac{1}{2}}\sum_{|k|\leq n}
	g_{n,k}(\omega){\bf b}_{n,k}(x)
\]
that for fixed $n\geq0$ and $(x,y)\in(\mathbb{S}^{2})^{2}$,
\begin{equation}
\label{eq:zero-g}
\mathbb{E}[e_{n}^{\omega}(x)|e_{n}^{\omega}(y)|^{2}] 
	= \mathbb{E}[e_{n}^{\omega}(x)\|e_{n}^{\omega}\|_{L^{2}(\mathbb{S}^{2})}^{2}] 
	= 0\,.
\end{equation}
For fixed $x\in\mathbb{S}^{2}$, we will use the fact that $e_{n}^{\omega}(x)$ is $\mathcal{B}_{n}$-measurable, where $\mathcal{B}_{n}$ is the $\sigma$-algebra generated by $(g_{n,k})_{|k|\leq n}$. The independence of $\mathcal{B}_{n}$ from $\mathcal{B}_{n'}$ whenever $n\neq n'$ plays an important role in the multilinear estimates involving the random functions $e_{n}^{\omega}$. 

\subsection{$L^{p}$-bounds}

We recall the eigenfunctions estimates due to Sogge \cite{sogge0}. There exists $C>0$ such that for all $n\geq0$ and $f\in L^{2}(\mathbb{S}^{2})$, 
\begin{equation}
\label{eq:Sogge} 
\|\pi_{n}f\|_{L^p(\mathbb{S}^2)}
\leq C\|\pi_{n}f\|_{L^{2}(\mathbb{S}^{2})} 
	\begin{cases}
			&\!\!\!\!\!\! \lambda_{n}^{\frac{1}{2}(\frac{1}{2}-\frac{1}{p})},\; 2\leq p\leq 6   \\
			&\!\!\!\!\! \lambda_{n}^{\frac{1}{2}-\frac{2}{p}},\; 6\leq p\leq\infty.
	\end{cases}
\end{equation}

In contrast, we prove the $L^{p}(\mathbb{S}^{2})$-norm of a $L^{2}(\mathbb{S}^{2})$-normalized Gaussian spherical harmonic $e_{n}^{\omega}$ has its moment of order $p$ uniformly bounded in $n$. 

\begin{lemma}\label{lem:Lp} There exists $C>0$ such that for all $p\geq 2$ and $n\geq0$,
\begin{equation}
    \label{eq:Lp}
  	  \|e_{n}^{\omega}\|
    	_{L_{\omega}^{p}(\Omega ; L_{x}^{p}(\mathbb{S}^{2}))}
    	 \leq C\sqrt{p}\,.
\end{equation}
\end{lemma}

\begin{proof} We deduce from Lemma \ref{lem:weyl} and from Khintchine’s inequality that for there exists $C>0$ such that for all $n\geq0$, $x\in\mathbb{S}^{2}$, $e_{n}^{\omega}(x)$ and $p\geq2$,
\[
  \|e_n^{\omega}(x)\|_{L^p_{\omega}(\Omega)} \leq C \sqrt{p}\,. 
\]
Subsequently, we deduce from the Fubini's Theorem that  
\[
\mathbb{E}
	\big[
	\|e_{n}^{\omega}\|_{L_{x}^{p}(\mathbb{S}^{2})}^{p}
	\big]^{\frac{1}{p}} 
	= 
	\big(
	\int_{\mathbb{S}^{2}}\|e_{n}^{\omega}(x)\|_{L_{\omega}^{p}}^{p}\mathrm{d}\sigma(x)
	\big)^{\frac{1}{p}} 
	\leq C\operatorname{vol}(\mathbb{S}^{2})^{\frac{1}{p}}\sqrt{p} = C\sqrt{p}\,,
\]
since we normalized the Lebesgue measure such that $\operatorname{vol}(\mathbb{S}^{2})=1$. This concludes the proof of Lemma~\eqref{lem:Lp}.
\end{proof}
In \cite{BurqLebeau14} the authors prove some more precise large deviation estimates, using the concentration of the measure. 
\subsection{Decomposition of the nonlinearity}

We decompose the Wick cubic power into three parts: 
\begin{align*}
:|u|^{2}u:&=|u|^2u-2\|u\|_{L^2(\mathbb{S}^{2})}^2u \\
	&= \mathcal{N}_{1}(u) + \mathcal{N}_{2}(u) + \mathcal{N}_{3}(u)
\end{align*}
where
\begin{align*}
	&\mathcal{N}_1(u)=\sum_{n_1,n_2,n_3}\mathbf{1}_{n_2\neq n_1,n_2\neq n_3 }\pi_{n_1}u\pi_{n_{2}}\overline{u}\pi_{n_3}u, \\
	&\mathcal{N}_2(u)=2\sum_{n_1,n_2}\mathbf{1}_{n_1\neq n_2}\big(|\pi_{n_{2}}u|^{2}-\|\pi_{n_{2}}u\|_{L^{2}(\mathbb{S}^{2})}^{2}\big)\pi_{n_1}u\\
	&\mathcal{N}_3(u)=\sum_{n}|\pi_{n}u|^{2}\pi_{n}u-2\|\pi_{n}u\|_{L^{2}(\mathbb{S}^{2})}^{2}\pi_{n}u\,.
\end{align*}

We isolate the resonant interaction from $\mathcal{N}_{2}$:
\begin{equation}
\label{eq:Nres}
\mathcal{N}_{2,\mathrm{res}}(u) := 2\sum_{n_{1},n_{2}}\mathbf{1}_{n_{1}\neq n_{2}}\pi_{n_{1}}
	\big(
	\pi_{n_{1}}u (|\pi_{n_{2}}u|^{2}-\|\pi_{n_{2}}u\|_{L^{2}(\mathbb{S}^{2})}^{2})
	\big)\,.
\end{equation}
As we will see in the next Section, the term $\mathcal{N}_{2,\mathrm{res}}$ is responsible for the divergence of the second Picard iteration \eqref{eq:FirstIterateSphere} claimed in our main Theorem.

\begin{remark}\label{rem:wick} In the case of the torus $\T^{2}$ where the plane waves $e^{in\cdot x}$ have constant amplitude, we have that for all $n\in\Z^{2}$ and $x\in\mathbb{T}^{2}$
\[
|e^{in\cdot x}|^{2} - \|e^{in\cdot x}\|_{L^{2}(\mathbb{T}^{d})}^{2} = 0\,,
\]
 so that the term $\mathcal{N}_{2}$ does not exit after the Wick ordering. On the sphere, however, the variable $x$ in the physical space has a role to play and pointwise in $x$,  the \emph{Wick square} 
 \[
 |e_{n}^{\omega}(x)|^{2} - \|e_{n}^{\omega}\|_{L^{2}(\mathbb{S})}^{2}
 \]
 has no reason to vanish. Still, it does vanish on average (both in space and in probability measure):
 \[
 \int_{\mathbb{S}^{2}} |e_{n}^{\omega}(x)|^{2} - \|e_{n}^{\omega}\|_{L^{2}(\mathbb{S})}^{2}\mathrm{d}\sigma(x) = \int_{\Omega} |e_{n}^{\omega}(x)|^{2} - \|e_{n}^{\omega}\|_{L^{2}(\mathbb{S})}^{2}\mathrm{d}\mathbb{P}(\omega) = 0\,.
 \]
\end{remark}

\subsection{The second Picard iteration}\label{sec:orga} Fix $\alpha\in\R$, and recall that the Gaussian measure $\mu_{\alpha}$ is induced by the random variable
\[
\phi_{\alpha}^{\omega} = \sum_{n\geq0}\lambda_{n}^{-(\alpha-\frac{1}{2})}e_{n}^{\omega}\,,
\]
where the Gaussian spherical harmonics $e_{n}^{\omega}$ are defined in \eqref{eq:def_en}. We denote $u^{\omega}(t)$ the linear evolution 
\begin{equation}
\label{eq:lin}
u^{\omega}(t,x) = \e^{it(\Delta-1)}\phi_{\alpha}^{\omega} = \sum_{n\geq0} \e^{-it\lambda_{n}^{2}}\lambda_{n}^{-(\alpha-\frac{1}{2})}e_{n}^{\omega}(x)\,.
\end{equation}
In order to make sense to the series in $L^{\infty}(\R;H^{s}(\mathbb{S}^{2}))$ for any $s\in\R$, we also consider the finite dimensional approximation 
\[
P_{\leq N}u^{\omega}(t,x) = \sum_{n\ :\ \lambda_{n}\leq N}e^{-it\lambda_{n}^{2}}\lambda_{n}^{-(\alpha-\frac{1}{2})}e_{n}^{\omega}(x)\,.
\]
For $N\in\N$, dyadic, set 
\[
\Gamma_{N} = \Big\{\vec{n}=(n_{0},n_{1},n_{2},n_{3})\in\N^{4}\ |\ n_{i}\leq N\ \text{for}\ i\in\{1,2,3\}\Big\}
\]
and 
\begin{align*}
&\Gamma_{N}^{(1)} := \{\vec{n}\in\Gamma_{N}\quad |\quad n_{2}\neq n_{1}\,,\quad n_{2}\neq n_{3} \}\,,\\
&\Gamma_{N}^{(2)}:=\{\vec{n}\in\Gamma_{N}\quad |\quad n_{2}=n_{3}\,,\quad n_{2}\neq n_{1}\}\,,\\
&\Gamma_{N}^{(3)}:= \{\vec{n}\in\Gamma_{N}\quad |\quad n_{1}=n_{2}=n_{3}\}\,.
\end{align*}
For $\vec{n}\in\Gamma_{N}$ we denote the resonance function 
\[
\Omega(\vec{n}) = \lambda_{n_{0}}^{2}-\lambda_{n_{1}}^{2}+\lambda_{n_{2}}^{2}-\lambda_{n_{3}}^{2}\,.
\]
Then, we have 
\begin{align}
\nonumber\mathcal{I}_{\mathbb{S}^{2}}(t,P_{\leq N}u^{\omega}) &=-i \int_{0}^{t}e^{i(t-t')(\Delta-1)}\mathcal{N}(P_{\leq N}u^{\omega}(t'))\mathrm{d}t' \\
\nonumber	&=-i\int_{0}^{t}e^{i(t-t')(\Delta-1)}
	\big(
	\mathcal{N}_{1}(P_{\leq N}u^{\omega}(t'))+\mathcal{N}_{2}(P_{\leq N}u^{\omega}(t'))+\mathcal{N}_{3}(P_{\leq N}u^{\omega}(t'))
	\big)
	\mathrm{d}t'\\
\label{eq:duhamel}	&=\mathrm{I} + \mathrm{II} + \mathrm{III}\,,
\end{align}
where 
\begin{align}
	\label{eq:I}&\mathrm{I}_{N}(t,x) 
	= 
	\sum_{\vec{n}\in \Gamma_{N}^{(1)}}
	\e^{-it\lambda_{n_{0}}^{2}}
	\big(
	\frac{\e^{-it\Omega(\vec{n})}-1}{\Omega(\vec{n})}
	\big)
	(\lambda_{n_{1}}\lambda_{n_{2}}\lambda_{n_{3}})^{-(\alpha-\frac{1}{2})}\pi_{n_{0}}
	\big(e_{n_{1}}^{\omega}\overline{e}_{n_{2}}^{\omega}e_{n_{3}}^{\omega}
	\big)\,, \\
	\label{eq:II}
	&\mathrm{II}_{N}(t,x)
	= \sum_{\vec{n}\in\Gamma_{N}^{(2)}}
	\e^{-it\lambda_{n_{0}}^{2}}
	\big(
	\frac{\e^{-it\Omega(\vec{n})}-1}{\Omega(\vec{n})}
	\big)
	(\lambda_{n_{1}}\lambda_{n_{2}}^{2})^{-(\alpha-\frac{1}{2})}\pi_{n_{0}}
	\big(e_{n_{1}}^{\omega}(|e_{n_{2}}^{\omega}|^{2}-\|e_{n_{2}}^{\omega}\|_{L^{2}\mathbb{S}^{2}}^{2})
	\big) \\
	\label{eq:III}
	&\mathrm{III}_{N}(t,x) 
	= \sum_{\vec{n}\in\Gamma_{N}^{(3)}}
	\e^{-it\lambda_{n_{0}}^{2}}
	\big(
	\frac{\e^{-it\Omega(\vec{n})}-1}{\Omega(\vec{n})}
	\big)
	\lambda_{n_{1}}^{-3(\alpha-\frac{1}{2})}\pi_{n_{0}}
	\big(
	|e_{n_{1}}^{\omega}|^{2}e_{n_{1}}^{\omega}
	\big) 
	- 
	2\sum_{n\geq0}\e^{-it\lambda_{n_{0}}^{2}}\lambda_{n}^{-3(\alpha-\frac{1}{2})}\|e_{n}^{\omega}\|_{L^{2}(\mathbb{S}^{2})}^{2}e_{n}^{\omega}\,.
\end{align}
\section{Proof of the main Theorem}
\label{sec:proof}
Given $\alpha\in\R$ our goal is to show a logarithmic divergence in $N$ of the $H^{\alpha-1}(\mathbb{S}^{2})$-norm of \eqref{eq:duhamel}. 

 \subsection{Isolating the divergent term}
We first isolate the singular contribution \eqref{eq:II} from \eqref{eq:I} by using probabilistic independence, and the ingredients such as \eqref{eq:zero-g} discussed at the end of Section~\ref{sec:gm}.
\begin{lemma}
\label{lem:iso} For every $\alpha$, $N$ and $t$, we have 
\begin{equation}
\label{eq:iso}
\Big|\|\mathrm{II}_{N}(t)\|_{L^{2}(\Omega ; H^{\alpha-1}(\mathbb{S}^{2}))}-\|\mathrm{III}_{N}(t)\|_{L^{2}(\Omega ; H^{\alpha-1}(\mathbb{S}^{2}))}\Big|
	 \leq \|\mathcal{I}_{\mathbb{S}^{2}}(t,P_{\leq N}u^{\omega})\|_{L^{2}(\Omega ; H^{\alpha-1}(\mathbb{S}^{2}))}\,,
\end{equation}
where the notations are introduced in Section~\ref{sec:orga}
\end{lemma}

\begin{proof} For simplicity we consider the case $\alpha=1$. The other cases are the same up to applying $\langle\Delta\rangle^{\frac{\alpha-1}{2}}$ to every term. We have 
\begin{multline*}
\|\mathcal{I}_{\mathbb{S}^{2}}(t,P_{\leq N}u^{\omega})\|_{L^{2}(\Omega ; L^{2}(\mathbb{S}^{2}))}^{2} = \|\mathrm{I}_{N}(t)\|_{L^{2}(\Omega ; L^{2}(\mathbb{S}^{2}))}^{2} + \|\mathrm{II}_{N}(t)+\mathrm{III}_{N}(t)\|_{L^{2}(\Omega ; L^{2}(\mathbb{S}^{2}))}^{2} \\
+ 2\mathbb{E}[\langle \mathrm{I}_{N}(t)|\mathrm{II}_{N}(t)+\mathrm{III}_{N}(t) \rangle_{L^{2}(\mathbb{S}^{2})}]\,.
\end{multline*}
We show that for every $t$ and $N$,
\begin{equation}
\label{eq:orth-I}
\mathbb{E}[\langle \mathrm{I}_{N}(t)|\mathrm{II}_{N}(t)\rangle_{L^{2}(\mathbb{S}^{2})}]=\mathbb{E}[\langle \mathrm{I}_{N}(t)|\mathrm{III}_{N}(t)\rangle_{L^{2}(\mathbb{S}^{2})}]=0\,.
\end{equation}
Expanding the scalar product, we have
\begin{multline*}
\mathbb{E}[\langle \mathrm{I}_{N}(t)|\mathrm{II}_{N}(t)\rangle_{L^{2}(\mathbb{S}^{2})}] = \sum_{\vec{n}\in\Gamma_{N}^{(1)}}\sum_{\vec{n}'\in\Gamma_{N}^{(2)}}\mathbf{1}_{n_{0}=n_{0}'}
	\mathbb{E}
	\Big[
	\big\langle \pi_{n_{0}}(e_{n_{1}}^{\omega}\overline{e}_{{n_{2}}}^{\omega}e_{n_{3}}^{\omega})\,
	\big|\,\overline{e}_{n_{1}'}^{\omega}(|e_{n'_{2}}^{\omega}|^{2}-\|e_{n'_{2}}^{\omega}
	\|_{L^{2}(\mathbb{S}^{2})}^{2})
	\big\rangle_{L^{2}(\mathbb{S}^{2})}\Big]\\
	\big(\frac{\e^{-it\Omega(\vec{n})}-1}{\Omega(\vec{n})}\big)
	(\lambda_{n_{1}}\lambda_{n_{2}}\lambda_{n_{3}})^{-(\alpha-\frac{1}{2})}
	\big(\frac{\e^{it\Omega(\vec{n}')}-1}{\Omega(\vec{n}')}
	\big)
	(\lambda_{n_{1}'}\lambda_{n_{2}'}^{2})^{-(\alpha-\frac{1}{2})}\,.
\end{multline*}
Using the integral kernel $K_{n_{0}}$ of $\pi_{n_{0}}$ we observe that for every $(\vec{n}, \vec{n}')\in(\Gamma_{N}^{(1)}\times \Gamma_{N}^{(2)})$,
\begin{multline*}
\mathbb{E}\Big[
	\big\langle \pi_{n_{0}}(e_{n_{1}}^{\omega}\overline{e}_{{n_{2}}}^{\omega}e_{n_{3}}^{\omega})\,
	\big|\,\overline{e}_{n_{1}'}^{\omega}(|e_{n_{2}}^{\omega}|^{2}-\|e_{n_{2}}^{\omega}\|_{L^{2}(\mathbb{S}^{2})}^{2})
	\big\rangle_{L^{2}(\mathbb{S}^{2})}\Big] \\
	= \int_{(\mathbb{S}^{2})^{2}}K_{n_{0}}(x,y)
	\mathbb{E}
	\big[
	e_{n_{1}}^{\omega}(x)\overline{e}_{n_{2}}^{\omega}(x)e_{n_{3}}^{\omega}(x)\overline{e}_{n_{1}'}^{\omega}(y)(|e_{n_{2}'}^{\omega}(y)|^{2}-\|e_{n_{2}'}^{\omega}\|_{L^{2}(\mathbb{S}^{2})}^{2})
	\big]
	\mathrm{d}\sigma(x)\mathrm{d}\sigma(y)\,.
\end{multline*}
We see from the non-pairing condition $n_{2}\neq n_{1}, n_{3}$ of $\vec{n}\in\Gamma_{N}^{(1)}$ and from \eqref{eq:zero-g} that, for fixed~$x$ and~$y$, 
\[
\mathbb{E}
	\Big[e_{n_{1}}^{\omega}(x)
	\overline{e}_{n_{2}}^{\omega}(x)
	e_{n_{3}}^{\omega}(x)
	\overline{e}_{n_{1}'}^{\omega}(y)
	(|e_{n_{2}'}^{\omega}(y)|^{2}-\|e_{n_{2}'}^{\omega}\|_{L^{2}(\mathbb{S}^{2})}^{2})\Big] = 0\,.
\]
This gives the first equality in \eqref{eq:orth-I}. The second follows analogously, using that for all $\vec{n}\in\Gamma_{N}^{(1)}$ and $\vec{n}'\in\Gamma_{N}^{(3)}$, and for all $x$, $y$ in $\mathbb{S}^{2}$,
\[
\mathbb{E}\Big[ e_{n_{1}}^{\omega}(x)\overline{e}_{n_{2}}^{\omega}(x)e_{n_{3}}^{\omega}(x) \overline{e}_{n_{1}'}^{\omega}(y)|e_{n'_{1}}^\omega(y)|^{2} \Big] = \mathbb{E}\Big[ e_{n_{1}}^{\omega}(x)\overline{e}_{n_{2}}^{\omega}(x)e_{n_{3}}^{\omega}(x) \overline{e}_{n_{1}'}^{\omega}(y)\|e_{n'_{1}}^\omega(y)\|_{L^{2}(\mathbb{S}^{2})}^{2} \Big] =0\,.
\]
This proves \eqref{eq:orth-I}. We deduce that 
\[
\|\mathcal{I}_{\mathbb{S}^{2}}(t,P_{\leq N}u^{\omega})\|_{L^{2}(\Omega ; L^{2}(\mathbb{S}^{2}))}^{2} = \|\mathrm{I}_{N}(t)\|_{L^{2}(\Omega ; L^{2}(\mathbb{S}^{2}))}^{2} + \|\mathrm{II}_{N}(t)+\mathrm{III}_{N}(t)\|_{L^{2}(\Omega ; L^{2}(\mathbb{S}^{2}))}^{2}\,,
\]
and we conclude from the triangle inequality that 
\begin{multline*}
\|\mathcal{I}_{\mathbb{S}^{2}}(t,P_{\leq N}u^{\omega})\|_{L^{2}(\Omega ; L^{2}(\mathbb{S}^{2}))} 
	\geq  \|\mathrm{II}_{N}(t)+\mathrm{III}_{N}(t)\|_{L^{2}(\Omega ; L^{2}(\mathbb{S}^{2}))}  
	\\
	\geq 
	\big|
	\|\mathrm{II}_{N}(t)\|_{L^{2}(\Omega ; H^{\alpha-1}(\mathbb{S}^{2}))}-\|\mathrm{III}_{N}(t)\|_{L^{2}(\Omega ; H^{\alpha-1}(\mathbb{S}^{2}))}
	\big|\,.
\end{multline*}
This concludes the proof of Lemma \ref{lem:iso}. 
\end{proof}

Let us now show that the term $\mathrm{III}_{N}$, written in \eqref{eq:III}, gains one derivative with respect to the initial data $\phi_{\alpha}$. This implies in particular that it is bounded in $L^{2}(\Omega ;H^{\alpha-1}(\mathbb{S}^{2}))$ uniformly in $N$. 
\begin{lemma}
\label{lem:III}
For every $\alpha$ and $s$ such that $s<3\alpha-2$, 
\[
\underset{N}{\sup}\ \mathbb{E}\Big[\|\mathrm{III}_{N}(t)\|_{H^{s}(\mathbb{S}^{2})}^2\Big]<+\infty\,. 
\]
In particular, when $\alpha > \frac{1}{2}$ then 
\[
\underset{N}{\sup}\ \mathbb{E}\Big[\|\mathrm{III}_{N}(t)\|_{H^{\alpha-1}(\mathbb{S}^{2})}^2\Big]<+\infty\,.
\]
\end{lemma}

\begin{proof}
Fix $N\in2^{\N}$. By Minkowski and the structure of $\mathrm{III}_N(t)$, we have 
\[
\Big( \mathbb{E}\Big[\|\mathrm{III}_{N}(t)\|_{H^{s}(\mathbb{S}^{2})}^2\Big]\Big)^{\frac{1}{2}} \leq \sum_{\substack{K\leq N\\ K\mathrm{dyadic}}}  \Big(\mathbb{E}\Big[\|\mathrm{III}(\operatorname{P}_{K}u^{\omega})(t)\|_{H^{s}(\mathbb{S}^{2})}^2\Big]\Big)^{\frac{1}{2}}\,.
\]
It suffices to show that there exists $\delta>0$, depending on $s$ and $\alpha$, such that for all $K$,
\[
 \mathbb{E}\Big[\|\mathrm{III}_N(\operatorname{P}_{K}u^{\omega})(t)\|_{H^{s}(\mathbb{S}^{2})}^2\Big]\lesssim K^{-\delta}\,.
\]
We have
\begin{multline}
\label{eq:III-s}
 \mathbb{E}
 	\Big[\|\mathrm{III}_N(\operatorname{P}_{K}u^{\omega})(t)\|_{H^{s}(\mathbb{S}^{2})}^2\Big] \\
 	=\sum_{n_{0}}\lambda_{n_{0}}^{2s}\mathbb{E}
	\Big[
	\big\|\pi_{n_{0}}\sum_{n_{1}\sim K}
	\big(
	\frac{\e^{-it(\lambda_{n_{0}}^{2}-\lambda_{n_{1}}^{2})}-1}{\lambda_{n_{0}}^{2}-\lambda_{n_{1}}^{2}}
	\big)\lambda_{n_{1}}^{-3(\alpha-\frac{1}{2})}|e_{n_{1}}^{\omega}|^{2}e_{n_{1}}^{\omega}
	\big\|_{L^{2}(\mathbb{S}^{2})}^{2}
	\Big]\,.
\end{multline}
By some degree considerations we see that the terms contribute only when $n_{0}\leq 3n_{1}$. Moreover, expanding the square we see that 
\begin{multline*}
|\eqref{eq:III-s}| 
	\lesssim \sum_{\lambda_{n_{0}}\lesssim K}\sum_{n_{1},n_{1}'\sim K}\lambda_{n_{0}}^{2s}(\lambda_{n_{1}}\lambda_{n_{1}'})^{-3(\alpha-\frac{1}{2})}\\
	\Big|\int_{(\mathbb{S}^{2})^{2}}K_{n_{0}}(x,y)\mathbb{E}
	[e_{n_{1}}^{\omega}(x)|e_{n_{1}}^{\omega}(x)|^{2}\overline{e}_{n_{1}'}^{\omega}(y)|e_{n_{1}'}^{\omega}(y)|^{2}] 
	\mathrm{d}\sigma(x)\mathrm{d}\sigma(y)\Big|
\end{multline*}
Using the independence of $e_{n_{1}}(x)$ and $e_{n_{1}'}(y)$ when $n_{1}\neq n_{1}'$, we see that only the terms with $n_{1}=n_{1}'$ contribute. Hence, 
\begin{align*}
|\eqref{eq:III-s}|
	&\lesssim K^{2s}\sum_{n_{0}\lesssim K}\sum_{n_{1}\sim K}\lambda_{n_{1}}^{-6(\alpha-\frac{1}{2})}\Big|\mathbb{E}\Big[
	\int_{(\mathbb{S}^{2})^{2}}K_{n_{0}}(x,y)|e_{n_{1}}^{\omega}(x)|^{2}|e_{n_{1}}^{\omega}(y)|^{2}e_{n_{1}}^{\omega}(x)\overline{e}_{n_{1}}^{\omega}(y)
	\mathrm{d}\sigma(x)\mathrm{d}\sigma(y)\Big]\Big| \\
	&\lesssim K^{2s}\sum_{n_{0}\lesssim K}\sum_{n_{1}\sim K}\lambda_{n_{1}}^{-6(\alpha-\frac{1}{2})} \mathbb{E}\|\pi_{n_{0}}(|e_{n_{1}}^{\omega}|^{2}e_{n_{1}}^{\omega})\|_{L^{2}(\mathbb{S}^{2})}^{2}] \\
	&\lesssim K^{2s}\sum_{n_{1}\sim K}\lambda_{n_{1}}^{-6(\alpha-\frac{1}{2})} \mathbb{E}\|e_{n_{1}}^{\omega}\|_{L^{6}(\mathbb{S}^{2})}^{6}\,.
\end{align*}
According to \eqref{eq:Lp} we conclude that 
\[
|\eqref{eq:III-s}| \lesssim K^{2s}\sum_{n_{1}\sim K}\lambda_{n_{1}}^{-6(\alpha-\frac{1}{2})} \lesssim K^{2s+1-6(\alpha-\frac{1}{2})}\,,
\]
which is conclusive when 
\[
s<3\alpha-2 = \alpha - 1 + 2(\alpha-\frac{1}{2})\,.
\]
This concludes the proof of Lemma \ref{lem:III}. 
\end{proof}

At this stage, we proved in Lemma~\ref{lem:iso} and in Lemma~\ref{lem:III} that there exists $C>0$ such that for all $N\geq 1$, $\alpha>\frac{1}{2}$,  $t\in\R$ and $s\leq 3\alpha - 2$,
\begin{equation}
\label{eq:redu}
\|\operatorname{II}_{N}(t)\|_{L^{2}(\Omega; H^{\alpha-1}(\mathbb{S}^{2}))} - C \leq \|\mathcal{I}_{\mathbb{S}^{2}}(t,P_{\leq N}u_{\alpha}^{\omega})\|_{L^{2}(\Omega; H^{\alpha-1}(\mathbb{S}^{2}))}\,.
\end{equation}
It remains to show the divergence of the term on the left-hand-side. 
\subsection{Proof of the divergence claimed in Theorem~\ref{th:main}}

In order to demonstrate the divergence asserted in Theorem~\ref{th:main}, we exploit the concentration property associated with high-order spherical harmonics, as written in Proposition~\ref{prop:concentration}.
\begin{proposition}
\label{prop:II}
There exist $\eta>0$ and $N_0\geq 0$ such that for all $N\geq N_0$ and $t\in\R$ we have
\begin{equation}
\label{eq:propII}
|t|\eta\log(N)^{\frac{1}{2}}
	\leq
	\|\operatorname{II}_N(t)
	\|_{L^2(\Omega ; H^{\alpha-1}(\mathbb{S}^{2}))}\,.
\end{equation}

\end{proposition}
\begin{proof}
Recall that we defined in \eqref{eq:II}
\[
\mathrm{II}_{N}(t,x)
	= \sum_{\vec{n}\in\Gamma_{N}^{(2)}}
	\big(\frac{\e^{-it\Omega(\vec{n})}-1}{\Omega(\vec{n})}
	\big)(\lambda_{n_{1}}\lambda_{n_{2}}^{2})^{-(\alpha-\frac{1}{2})}\pi_{n_{0}}
	\big(e_{n_{1}}^{\omega}(|e_{n_{2}}^{\omega}|^{2}-\|e_{n_{2}}^{\omega}\|_{L^{2}\mathbb{S}^{2}}^{2})
	\big)\,.
\]
Note that when $\vec{n}=(n,n_{1},n_{2},n_{3})\in\Gamma_{N}^{(2)}$ then the resonant function is $\Omega(\vec{n})=\lambda_{n}^{2}-\lambda_{n_{1}}^{2}$. We expand $\mathrm{II}_{N}(t)$ on the orthonormal basis of $L^2(\mathbb{S}^{2})$ made of spherical harmonics : 
\begin{multline*}
\mathrm{II}_{N}(t)=
\sum_{n\lesssim N}\sum_{|k|\leq n}Y_{n,k}\\
	\sum_{n_1,n_2\leq N}
	\mathbf{1}_{n_{1}\neq n_{2}}
	(\lambda_{n_{1}}
	\lambda_{n_{2}}^{2})^{-(\alpha-\frac{1}{2})}
	\frac{e^{-it(\lambda_{n}^2-\lambda_{n_1}^2)}-1}
	{\lambda_{n}^2-\lambda_{n_1}^2}
		\big(
	\int_{\mathbb{S}^{2}}
	e_{n_{1}}^{\omega}(x)
	(|e_{n_2}^\omega(x)|^2-\|e_{n_2}^\omega\|_{L^2(\mathbb{S}^{2})}^2)
	\overline{Y_{n,k}(x)}\mathrm{d}\sigma(x)
	\big)
	\,.
\end{multline*}

\begin{remark} Note that since we truncated the initial data at frequency $N$, we can deduce from degree considerations on the spherical harmonics that only the modes with $n\lesssim N$ contribute to the above sum. 
\end{remark}

By applying the Plancherel's formula and using the fact that $\e^{it(\Delta_{\mathbb{S}^2}-1)}$ is a unitary operator on $L^2(\mathbb{S}^{2})$, we obtain
\begin{multline}
\label{eq:II1}
	\|
	\operatorname{II}_{N}(t)
	\|_{H^{\alpha-1}(\mathbb{S}^{2})}^2
	= 
	\sum_{n\lesssim N}
	\lambda_{n}^{2(\alpha-1)}
	\sum_{|k|\leq n}
	\Big|
	\sum_{n_1,n_2\leq N}\mathbf{1}_{n_{1}\neq n_{2}}
	(\lambda_{n_{1}}
	\lambda_{n_{2}}^{2})^{-(\alpha-\frac{1}{2})}
	\frac{e^{-it(\lambda_{n}^2-\lambda_{n_1}^2)}-1}{\lambda_{n}^2-\lambda_{n_1}^2} \\
		\big(
	\int_{\mathbb{S}^{2}}
	e_{n_{1}}^{\omega}(x)
	(|e_{n_2}^\omega(x)|^2-\|e_{n_2}^\omega\|_{L^2(\mathbb{S}^{2})}^2)
	\overline{Y_{n,k}(x)}\mathrm{d}\sigma(x)
	\big)
	\Big|^2.
\end{multline}
%%%%%%%%%%%%%%%%%%%%%%%%%%%%%%%%%%%%%%%%%%
By expressing $e_{n_{1}}^{\omega}$ in the basis made of spherical harmonics $(Y_{n_{1},k_{1}})$ (recall that the law of $e_{n_{1}}^{\omega}$ does not depend on the choice of the orthonormal basis), one deduces from the independence and from~\eqref{eq:zero-g} that
\begin{multline}
\label{eq:II2}
\mathbb{E}\Big|
	\sum_{n_1,n_2\leq N}
	\mathbf{1}_{n_{1}\neq n_{2}}
	(\lambda_{n_{1}}
	\lambda_{n_{2}}^{2})^{-(\alpha-\frac{1}{2})}
	\frac{e^{-it(\lambda_{n}^2-\lambda_{n_1}^2)}-1}{\lambda_{n}^2-\lambda_{n_1}^2}
		\big(
	\int_{\mathbb{S}^{2}}
	e_{n_{1}}^{\omega}(x)
	(|e_{n_2}^\omega(x)|^2-\|e_{n_2}^\omega\|_{L^2(\mathbb{S}^{2})}^2)
	\overline{Y_{n,k}(x)}\mathrm{d}\sigma(x)
	\big)
	\Big|^2
\\
	= \sum_{n_{1},n_{2}\leq N} \mathbf{1}_{n_{1}\neq n_{2}}
	(\lambda_{n_{1}}\lambda_{n_{2}}^{2})^{-2(\alpha-\frac{1}{2})}(2n_{1}+1)^{-1}
	\big|\frac{e^{-it(\lambda_{n}^2-\lambda_{n_1}^2)}-1}{\lambda_{n}^2-\lambda_{n_1}^2}
	\big|^{2}
	\\
	\sum_{|k_{1}|\leq n_{1}}
	\mathbb{E}
	\Big|
	\int_{\mathbb{S}^{2}}
	(|e_{n_{2}}^{\omega}(x)|^{2}-\|e_{n_{2}}^{\omega}\|_{L^{2}(\mathbb{S}^{2})}^{2})
	g_{n_{1},k_{1}}(\omega)
	Y_{n_{1},k_{1}}(x)
	\overline{Y_{n,k}(x)}
	\mathrm{d}\sigma(x)
	\Big|^{2}
\end{multline}
Let us justify why only the terms with $n_{1}=n_{1}'$ and $n_{2}=n_{2}'$ contributed to the double sum over $(n_{1},n_{1}',n_{2},n_{2}')$.  Fix $(x,y)\in(\mathbb{S}^{2})^{2}$. We distinguish three cases: 
\medskip

\noindent {\bf Case 1: $n_{1}\neq n_{1}'$ and $n_{2}\neq n_{2}'$}. Since, by assumption, $n_{1}'\neq n_{2}'$ and $n_{1}\neq n_{2}$, we have from the independence of $\mathcal{B}_{n_{1}}$ and $\mathcal{B}_{n_{2}'}$ from $\mathcal{B}_{n_{1}'}$ and $\mathcal{B}_{n_{2}}$, that 
\begin{multline*}
\mathbb{E}\Big[
		e_{n_{1}}^{\omega}(x)
		\big(
		|e_{n_{2}}^{\omega}(x)|^{2}-\|e_{n_{2}}^{\omega}\|_{L^{2}(\mathbb{S}^{2})}^{2}
		\big)
		\overline{e_{n_{1}'}^{\omega}(y)}
		\big(
		|e_{n_{2}'}^{\omega}(y)|^{2}-\|e_{n_{2}'}^{\omega}\|_{L^{2}(\mathbb{S}^{2})}^{2}
		\big)
		\Big]
		\\
		=
		\mathbb{E}\Big[
		e_{n_{1}}^{\omega}(x)\big(
		|e_{n_{2}'}^{\omega}(y)|^{2}-\|e_{n_{2}'}^{\omega}\|_{L^{2}(\mathbb{S}^{2})}^{2}
		\big)
		\Big]
		\mathbb{E}\Big[
		\overline{e_{n_{1}'}^{\omega}(y)}
		\big(
		|e_{n_{2}}^{\omega}(x)|^{2}-\|e_{n_{2}}^{\omega}\|_{L^{2}(\mathbb{S}^{2})}^{2}
		\big)
		\Big]
		= 0\,,
\end{multline*}
where we used the identity \eqref{eq:zero-g}. 
\medskip

\noindent{\bf Case 2: $n_{1}\neq n_{1}'$ and $n_{2}=n_{2}'$}. In this case, we have from the independence of $\mathcal{B}_{n_{1}}$ from~$\mathcal{B}_{n_{1}'}$ and~$\mathcal{B}_{n_{2}}$ that 
\begin{multline*}
\mathbb{E}\Big[
		e_{n_{1}}^{\omega}(x)\big(
		|e_{n_{2}}^{\omega}(y)|^{2}-\|e_{n_{2}}^{\omega}\|_{L^{2}(\mathbb{S}^{2})}^{2}
		\big)
		\overline{e_{n_{1}'}^{\omega}(y)}
		\big(
		|e_{n_{2}}^{\omega}(y)|^{2}-\|e_{n_{2}}^{\omega}\|_{L^{2}(\mathbb{S}^{2})}^{2}
		\big)
		\Big] \\
		= 
		\mathbb{E}[
		e_{n_{1}}^{\omega}(x)
		]
		\mathbb{E}\Big[
		\big(
		|e_{n_{2}}^{\omega}(y)|^{2}-\|e_{n_{2}}^{\omega}\|_{L^{2}(\mathbb{S}^{2})}^{2}
		\big)
		\overline{e_{n_{1}'}^{\omega}(y)}\big(
		|e_{n_{2}}^{\omega}(y)|^{2}-\|e_{n_{2}}^{\omega}\|_{L^{2}(\mathbb{S}^{2})}^{2}
		\big)
		\Big]
		= 0\,,
\end{multline*}
where we used that $e_{n_{1}}^{\omega}(x)$ has zero average. 
\medskip

\noindent {\bf Case 3: $n_{1}=n_{1}'$ and $n_{2}\neq n_{2}'$.} By assumption, $n_{1}\neq n_{2}$ and we deduce from the independence of~$\mathcal{B}_{n_{2}}$ from~$\mathcal{B}_{n_{1}}$ and~$\mathcal{B}_{n_{2}'}$ that 
\begin{multline*}
\mathbb{E}\Big[
		e_{n_{1}}^{\omega}(x)
		\big(
		|e_{n_{2}}^{\omega}(x)|^{2}-\|e_{n_{2}}^{\omega}\|_{L^{2}(\mathbb{S}^{2})}^{2}
		\big)
		\overline{e_{n_{1}}^{\omega}(y)}
		\big(
		|e_{n_{2}'}^{\omega}(y)|^{2}-\|e_{n_{2}'}^{\omega}\|_{L^{2}(\mathbb{S}^{2})}^{2}
		\big)
		\Big]
		\\
		=
		\mathbb{E}\Big[
		|e_{n_{2}}^{\omega}(x)|^{2}-\|e_{n_{2}}^{\omega}\|_{L^{2}(\mathbb{S}^{2})}^{2}
		\Big]
		\mathbb{E}\Big[
		e_{n_{1}}^{\omega}(x)\overline{e_{n_{1}}^{\omega}(y)}
		\big(
		|e_{n_{2}'}^{\omega}(y)|^{2}-\|e_{n_{2}'}^{\omega}\|_{L^{2}(\mathbb{S}^{2})}^{2}
		\big)
		\Big]
		= 0\,, 
\end{multline*}
where we used that the first term in the second line vanishes according to \eqref{eq:zero-g}. 
\medskip

This proves that only $n_{1}=n_{1}'$ and $n_{2}=n_{2}'$ contributed to the double sum. 

\medskip

We now estimate from below \eqref{eq:II2}, keeping only the $high\ \times\ low$ frequency interactions, namely when $n_{2}=1$, which are resonant at $n=n_{1}$, and with $k=k_{1}$. Uniformly in $n,k$ with $n\geq2$, we have:
\begin{align*}
\eqref{eq:II2}
	&\gtrsim t^{2} n^{-1}
	\lambda_{n}^{-2(\alpha-\frac{1}{2})}\mathbb{E}
	\Big|
	\int_{\mathbb{S}^{2}}
	\big(|e_{1}^{\omega}(x)|^{2} - \|e_{1}^{\omega}\|_{L^{2}(\mathbb{S}^{2})}^{2}
	\big)
	|Y_{n,k}(x)|^{2}
	\mathrm{d}\sigma(x)\Big|^{2} \\
	&\sim t^{2} n^{-2}\lambda_{n}^{-2(\alpha-1)}
	\mathbb{E}\Big|
	\int_{\mathbb{S}^{2}}
	\big(|e_{1}^{\omega}(x)|^{2} - \|e_{1}^{\omega}\|_{L^{2}(\mathbb{S}^{2})}^{2}
	\big)
	|Y_{n,k}(x)|^{2}
	\mathrm{d}\sigma(x)\Big|^{2}\,,
\end{align*}
where we used that $\lambda_{n}\sim cn$.  We deduce from \eqref{eq:II} and \eqref{eq:II2} that 
\begin{equation}
\label{eq:II3}
\mathbb{E}\|\operatorname{II}_{N}(t)\|_{H^{\alpha-1}}^{2}
	\gtrsim 
	t^{2}\sum_{2\leq n\lesssim N}n^{-2}\sum_{|k|\leq n}
		\mathbb{E}\Big|
	\int_{\mathbb{S}^{2}}
	\big(|e_{1}^{\omega}(x)|^{2} - \|e_{1}^{\omega}\|_{L^{2}(\mathbb{S}^{2})}^{2}
	\big)
	|Y_{n,k}(x)|^{2}
	\mathrm{d}\sigma(x)\Big|^{2}\,.
\end{equation}
Then, we use the explicit expression of $e_1^\omega$, and we remove a well-chosen subset from $\Omega$ (such that the remainder still has positive measure) in order to exploit the concentration of $Y_{1,1}$ near the equator. We have, for $\omega\in\Omega$ and $x\in\mathbb{S}^{2}$, 
\[
e_1^\omega(x) 
	= \frac{1}{\sqrt{3}}
	\big(g_{1,1}(\omega)Y_{1,1}(x)+g_{1,-1}(\omega) Y_{1,-1}(x)+g_{1,0}(\omega) Y_{1,0}(x)
	\big)
\,,
\]
According to \eqref{eq:sh}, we can express in spherical coordinates
\[
Y_{1,1}(\theta,\varphi) = - \sqrt{\frac{3}{2}}\sin(\theta)\e^{i\varphi}\,. 
\]
Then, given $0<\epsilon\ll 1$ we define
\[
S_{\epsilon} = \set{\omega\in\Omega \mid 1\leq |g_{1,1}(\omega)|^2\leq100,\quad 
\abs{g_{1,-1}(\omega)}^2+\abs{g_{1,0}(\omega)}^2\leq\epsilon^{2}}.
\]
Note that $\mathbb{P}(S_{\epsilon})>0$. We write 
\begin{equation*}
|e_{1}^{\omega}(x)|^{2}-\|e_1^\omega\|_{L^2(\mathbb{S}^{2})}^2
	=\frac{1}{3}|g_{1,1}(\omega)|^2
	(\frac{3}{2}\sin^{2}(\theta)-1
	)
	+r(x,\omega)\,,
\end{equation*}
where there exists $C>0$ such that function $r$ satisfies for all $x\in\mathbb{S}^{2}$ and $\omega\in S_{\epsilon}$,
\begin{equation}
\label{eq:epsilon}
|r(x,\omega)|\leq C\epsilon\,.
\end{equation}
We observe that for all $\delta>0$, 
\begin{equation}
\label{eq:12}
x=(\theta,\varphi)\in\mathbb{S}^{2}\setminus\mathscr{C}_{2\delta}
	\quad \implies\quad
	 \frac{3}{2}\sin^{2}(\theta)-1\geq \frac{1}{2}(1-12\delta^{2})\,,
\end{equation}
where $\mathscr{C}_{\delta}$ is defined in \eqref{eq:Cdelta}. Fix $\delta>0$, $\omega\in S_{\epsilon}$ and set 
\begin{align*}
	a_{n,k}(\omega,\delta) &=
		\frac{1}{3}
		|g_{1,1}(\omega)|^2 \int_{\mathbb{S}^{2}}(1-\mathbf{1}_{\mathscr{C}_{2\delta}}(x)) 
	(\frac{3}{2}\sin^{2}(\theta)-1
	)
	\abs{Y_{n,k}(x)}^2
	\mathrm{d}\sigma(x)\\
	 \quad b_{n,k}(\omega,\delta) 
	 &=\frac{1}{3}
	 |g_{1,1}(\omega)|^2 \int_{\mathbb{S}^{2}}\mathbf{1}_{\mathscr{C}_{2\delta}}(x) 
	(\frac{3}{2}\sin^{2}(\theta)-1
	)
	\abs{Y_{n,k}(x)}^2
	\mathrm{d}\sigma(x)
	+
	 \int_{\mathbb{S}^{2}}r(x,\omega)\abs{Y_{n,k}(x)}^2
	 \mathrm{d}\sigma(x)\,.
\end{align*}
so that 
\[
\int_{\mathbb{S}^{2}}
	\big(|e_{1}^{\omega}(x)|^{2} - \|e_{1}^{\omega}\|_{L^{2}(\mathbb{S}^{2})}^{2}
	\big)
	|Y_{n,k}(x)|^{2}
	\mathrm{d}\sigma(x) = a_{n,k}(\omega,\delta) + b_{n,k}(\omega,\delta)\,.
\]
Then, we deduce from the inequality $ (a+b)^{2}\geq \frac{a^{2}}{2}-b^{2}$ (for $a,b\in\R$) that for all $\omega, n, k$,
\begin{multline}
\label{eq:II4}
	\mathbb{E}
	\Big[\mathbf{1}_{S_{\epsilon}}(\omega)
	\big|
	\int_{\mathbb{S}^{2}}
	\big(\abs{e_1^\omega(x)}^2-\|e_1^\omega\|_{L^2(\mathbb{S}^{2})}
	\big)|Y_{n,k}(x)|^2
	\mathrm{d}\sigma(x)
	\big|^{2}\Big]
	\\
	\geq \frac{1}{2}\mathbb{E}[\mathbf{1}_{S_{\epsilon}}(\omega)a_{n,k}(\omega)^{2}] - \mathbb{E}[\mathbf{1}_{S_{\epsilon}}(\omega)b_{n,k}(\omega)^{2}]\,.
\end{multline}
Moreover, for all $\epsilon,\delta>0$, $n,k$ and $\omega\in S_{\epsilon}$ we have form \eqref{eq:12} that 
\begin{align*}
a_{n,k}(\omega,\delta) 
	&\geq
	\frac{1}{6} |g_{1,1}(\omega)|^{2} (1-12\delta^{2})
	\int_{\mathbb{S}^{2}}
	(1-\mathbf{1}_{\mathscr{C}_{2\delta}}(x))|Y_{n,k}(x)|^{2}
	\mathrm{d}\sigma(x)\,. 
\intertext{On the other hand, it follows from \eqref{eq:epsilon} and from the fact that $\frac{3}{2}\sin(\theta)^{2}-1\leq\frac{1}{2}$ when $x=(\theta,\varphi)\in\mathscr{C}_{\epsilon}$, that }
|b_{n,k}(\omega,\delta)| 
	&\leq
	\frac{1}{12}|g_{1,1}(\omega)|^{2}\int_{\mathbb{S}^{2}}\mathbf{1}_{\mathscr{C}_{2\delta}}(x)|Y_{n,k}(x)|^{2}
	\mathrm{d}\sigma(x)
	+ 
	C\epsilon\,.
\end{align*}

\medskip

In section~\ref{section:equator}, we prove that in the regime where $n$ is large and $ n(n+1)(1-\delta^{2})\leq k^{2}\leq n^{2}$ (see \eqref{eq:k}) the spherical harmonics $\set{Y_{n,k}}$ concentrate their mass near the equator, precisely in the region outside $\mathscr{C}_{2\delta}$. Applying Proposition~\ref{prop:concentration}, we obtain that there exits $C_{\delta}$ such that for all $\delta>0$, $\omega\in S_{\epsilon}$, $n\geq2$ and $k$ as in \eqref{eq:k}, 
\begin{align*}
a_{n,k}(\omega)
	&\geq \frac{1}{6}|g_{1,1}(\omega)|^{2}(1-12\delta^{2})
		- \frac{C_{\delta}}{n}|g_{1,1}(\omega)|^{2}\,,
	\intertext{and}
|b_{n,k}(\omega)|
	&\leq
	 \frac{C_{\delta}}{n}|g_{1,1}(\omega)|^{2} + C\epsilon\,.
\end{align*} 
Let $\delta>0$ and $N\in\N$. We set 
\[
\Lambda_{N,\delta} := \{(n,k)\in\N^{2}\,,\quad N\leq n\,,\quad n(n+1)(1-\delta^{2})\leq k^{2}\leq n^{2}\}\,.
\]
We conclude from \eqref{eq:II4} that, when $\delta>0$ is sufficiently small, there exists $N_{\delta}\geq2$ such that for all $(n,k)\in\Lambda_{N_{\delta},\delta}$ and  $\epsilon>0$ sufficiently small,
\begin{multline*}
\mathbb{E}\Big|
	\int_{\mathbb{S}^{2}}
	\big(|e_{1}^{\omega}(x)|^{2} - \|e_{1}^{\omega}\|_{L^{2}(\mathbb{S}^{2})}^{2}
	\big)
	|Y_{n,k}(x)|^{2}
	\mathrm{d}\sigma(x)\Big|^{2}
	\\
	\geq 
	\mathbb{E}
	\Big[\mathbf{1}_{S_{\epsilon}}(\omega)
	\big|
	\int_{\mathbb{S}^{2}}
	\big(|e_{1}^{\omega}(x)|^{2} - \|e_{1}^{\omega}\|_{L^{2}(\mathbb{S}^{2})}^{2}
	\big)
	|Y_{n,k}(x)|^{2}
	\mathrm{d}\sigma(x)\big|^{2}
	\Big]
	\geq \frac{1}{100}\mathbb{P}(S_{\epsilon})\,.
\end{multline*}
Plugging the above estimate in \eqref{eq:II3} we conclude there exist $\delta>0$ and $N_{\delta}$ such that for all $N\geq N_{\delta}$ and $\epsilon>0$ small enough, 
\begin{align*}
\mathbb{E}\|\operatorname{II}_{N}(t)\|_{H^{\alpha-1}(\mathbb{S}^{2})}^{2} \gtrsim t^{2}
	\sum_{N_{\delta}\leq n\leq N}n^{-2}
	\#\{k\ :\ (n,k)\in\Lambda_{N_{\delta},\delta}\}
	\frac{1}{100}\mathbb{P}(S_{\epsilon}) \gtrsim_{\delta,\epsilon} t^{2}\log(N)\,.
\end{align*}
At this stage, $\delta$ and $\epsilon$ are small but fixed positive constants. This completes the proof of Proposition~\ref{prop:II}. 
\end{proof}
\subsection{Conclusion}

In order to complete the proof of Theorem \ref{th:main}, it suffices to inject~\eqref{eq:propII} into~\eqref{eq:redu}. This yields the existence of $C,\eta>0$ and a sufficiently large integer $N_{0}$ such that for all $t$ and for all $N\geq N_{0}$, 
\[
\|\mathcal{I}_{\mathbb{S}^{2}}(t,P_{\leq N}u_{\alpha}^{\omega})\|_{L^{2}(\Omega;H^{\alpha-1}(\mathbb{S}^{2}))}
	\geq |t|\eta \log(N)^{\frac{1}{2}} - C
\]
Taking $N_{0}$ large enough compared to $C$ gives 
\[
\|\mathcal{I}_{\mathbb{S}^{2}}(t,P_{\leq N}u_{\alpha}^{\omega})\|_{L^{2}(\Omega;H^{\alpha-1}(\mathbb{S}^{2}))}
	\geq |t|\frac{\eta}{2}\log(N)^{\frac{1}{2}}\,.
\]
This completes the proof of Theorem \ref{th:main}.
\section{Regularity of the first iteration on tori}
\label{sec:appendix}

In this appendix, we consider the cubic NLS on the general torus \[\mathbb{T}_{\beta}^2=\R^2/(2\pi\Z)^2\,,\]
endowed with the metric  $g=dx_1^2+\beta^{-2}dx_2^2$ for some $\beta>0$. The cubic NLS on $\T_{\beta}^2$ is written as
\begin{align}\label{NLSirrational}
 i\partial_{t}v+\Delta_{\beta}v=|v|^2v,
\end{align}
where  $\Delta_{\beta}=\partial_{x_1}^2+\beta^2\partial_{x_2}^2$. Hence from now on we will only concentrate on \eqref{NLSirrational}. For $n=(k,m),n'=(k',m')\in\Z^2$, we denote by $Q(n,n')=kk'+\beta^2 mm'$ the associate quadratic form and $Q(n):=Q(n,n)$. Then 
\[
e^{it\Delta_{\beta}}=\mathcal{F}_x^{-1}e^{-itQ(\cdot)}\mathcal{F}_x\,.
\]
In light of the Gaussian measure \eqref{eq:randomData} defined in the case of $\mathbb{S}^{2}$, we consider initial data of the form
\[\phi_{1}^\omega (x) = \sum_{n\in\Z^2}\frac{1}{\langle n \rangle}g_n^\omega\e^{in\cdot x}\,,
\]
where $\langle n\rangle =(1+Q(n))^{\frac{1}{2}}$, and
\[
u^\omega(t,x)=\sum_{n\in\Z^2}\frac{\e^{itQ(n)}}{\langle n \rangle}g_n^\omega\e^{in\cdot x}
\]
We write the Wick-ordered nonlinearity 
 \begin{align*}
 \mathcal{N}(u^\omega)(t)&= \sum_{\substack{(n_1,n_2,n_3)\in\Z^2\\n_2\neq n_1,n_3}}\frac{\e^{it(Q(n_1)-Q(n_2)+Q(n_3))}}{\langle n_1 \rangle\langle n_2 \rangle\langle n_3 \rangle}g_{n_1}(\omega)\overline{g_{n_2}(\omega)}g_{n_3}(\omega)\e^{i(n_1-n_2+n_3)\cdot x}\,\\
 &=\sum_{\substack{(n_1,n_2,n_3,n)\in\Gamma} } \frac{\e^{it\Phi(n)}}{\langle n_1 \rangle\langle n_2 \rangle\langle n_3 \rangle}g_{n_1}(\omega)\overline{g_{n_2}(\omega)}g_{n_3}(\omega) \e^{in\cdot x},
 \end{align*}
where the constraint set is 
\[\Gamma=\{(n_1,n_2,n_3,n)\in(\Z^2)^4: n=n_1-n_2+n_3, n_2\neq n_1,n_2\neq n_3 \}\,.
\]
and the resonant function $\Phi$ defined on $\Gamma$ is
\[
    \Phi(n):=Q(n_1)-Q(n_2)+Q(n_3)-Q(n_1-n_2+n_3)=2Q(n_1-n_2,n_2-n_3).
\]
The second Picard iteration is written as
\[
\mathcal{I}_{\mathbb{T}_\beta^2}(t,\phi^\omega)=\int_0^te^{i(t-t')\Delta_{\beta}}\mathcal{N}(u^\omega(t'))dt',
\]
We shall now prove the gain of regularity~\eqref{eq:FirstIterateTori} claimed in Theorem~\ref{th:main}.
\begin{proof} 
For fixed $t$, direct computation, independence assumption as well as the observation that $\mathbb{E}[g^2]=0$ for standard complex guassian functions yields that
\begin{equation}\label{iteration1-1}
    \begin{split}
    \mathbb{E}\left[\|\mathcal{I}_{\mathbb{T}_\beta^2}(t,\phi^\omega)\|_{H_x^s}^2 \right] \sim \sum_{n=(n_1,n_2,n_3)\in\Gamma}\frac{\langle n_1-n_2+n_3\rangle^{2s} }{\langle n_1\rangle^2\langle n_2\rangle^2\langle n_3\rangle^3\langle \Phi(n)\rangle^2 },
    \end{split}
\end{equation}
To compute the right side of \eqref{iteration1-1}, we decompose dyadically $|n_j|\in N_j$, and assume that $N^{(1)}\geq N^{(2)}\geq N^{(3)}$ is the non-increasing order of $N_1,N_2,N_3$. It would be sufficient to obtain an inequality of the form
\begin{align}\label{A2}
    (N_1N_2N_3)^{-2}\sum_{\substack{|n_j|\sim N_j, j=1,2,3\\
    n_2\neq n_1,n_3 } }\frac{\langle n_1-n_2+n_3\rangle^{2s} }{\langle Q(n_1-n_2,n_2-n_3) \rangle^2 }\lesssim (N^{(1)})^{-\delta}.
\end{align}
The left side of \eqref{A2} can be bounded by
\begin{align}\label{iteration1-2}
    &(N_1N_2N_3)^{-2}\sum_{l\in\Z}\frac{(N^{(1)})^{2s}}{\langle l\rangle^2}\sum_{\substack{|n_j|\in N_j\\
    n_2\neq n_1,n_3} }\mathbf{1}_{|Q(n_1-n_2,n_2-n_3)-l|\leq\delta }\\ \leq &(N_1N_2N_3)^{-2}(N^{(1)})^{2s}\sup_{ l\in\Z}\sum_{\substack{|n_j|\in N_j\\
    n_2\neq n_1,n_3} }\mathbf{1}_{|Q(n_1-n_2,n_2-n_3)-l|\leq\delta },
\end{align}
for some number $0<\delta<1$ to be fixed later.
\medskip

Let us now estimate the quantity (for fixed $l\in\Z$ and $\delta<1/4$)
\begin{align}\label{counting1}
    M_{N_1,N_2,N_3}:=\sum_{\substack{|n_j|\sim N_j\\ n_2\neq n_1,n_3}}\mathbf{1}_{|Q(n_1-n_2,n_2-n_3)-l|\leq \delta }:
\end{align}
\medskip

\noindent {\bf Case 1: $N_1\gg N_2,N_3$}. In this case, \eqref{counting1} can be majorized by
\[
    (N_2N_3)^2\sup_{n_2,n_3:n_2\neq n_3}\sum_{\substack{n_1:|n_1|\sim N_1,n_1\neq n_2 }}\mathbf{1}_{  |Q(n_1-n_2,n_2-n_3)-l|\leq\delta}.
\]
For fixed  $n_2,n_3$, we denote by $m_1=n_1-n_2, m_0=n_2-n_3$, $|m_1|\sim N_1\gg |m_0|$. It is sufficient to estimate the number of $m_1\in \Z^2, |m_1|\sim N_1$ such that $|Q(m_1,m_0)-l|\leq\delta$. Denote by $m_0=(\xi_0,\eta_0)$, if $\eta_0=0$ or $\xi_0=0$, $Q(m_1,m_0)\in\beta^2\Z$ or $Q(m_1,m_0)\in\Z$, with respectively. Then $Q(m_1,m_0)$ can only take a discrete number of values. In these situations, we have (the same dimension treatment as in \cite{BourgainTrick96})
\[ 
    \#\{m_1\in\Z^2: |m_1|\sim N_1, |Q(m_1,m_0)-l|\leq\delta \} \lesssim N_1.
\]
Therefore, without loss of generality, we may assume that both $\xi_0$ and $\eta_0$ are non zero. Denote by $m_{\beta}=(\xi_0,\beta^2\eta_0)=\textrm{diag}(1,\beta^2)m_0$. Then it is reduced to estimate the cardinality of the set
\[
    S_{\delta}:=\left\{z\in\Z^2: \left|z\cdot \frac{m_{\beta}}{|m_{\beta}|}-\frac{l}{|m_{\beta}|}\right|\leq \frac{\delta}{|m_{\beta}|} \right\}.
\]
We observe that $S_{\delta}$ is a rectangle with side length $\sim N_1$ and width $\leq 2\delta$. Since $2\delta<1/2$, we see that
\[ 
    \#S_{\delta}\lesssim N_1.
\]
Thus the contribution of~\eqref{iteration1-2} in the sum is less then $N_1^{2s-1}$, and the associated dyadic summation over $N_1\gg N_2,N_3$ converges provided that $s<\frac{1}{2}$. 
\medskip

\noindent {\bf Case 2: $N_2\gg N_1,N_3$}. In this case,  $|Q(n_1-n_2,n_2-n_3)|\sim N_2^{2}$. Coming back to \eqref{iteration1-2}, the range of the sum of $l$ is $|l|\geq N_2^2$. Hence by the crude estimate $M_{N_1,N_2,N_3}\lesssim (N_1N_2N_3)^3$, the contribution of \eqref{iteration1-2} is bounded by
\[ \sum_{|l|\geq N_2^2 }\frac{N_2^{2s}}{l^2}\lesssim N_2^{2s-2}.
\]  
Then associated dyadic summation over $N_2\gg N_1, N_3$ converges, provided that $s<1$.

\medskip

For the rest situations, the argument is the same as for {\bf Case 1}. For example, if $N_1\sim N_2\gg N_3$ ($N_2\sim N_3\gg N_1 $), we can fix $n_2,n_3$ ($n_1,n_2$), and do the same manipulation as in {\bf Case 1}.

\medskip

This completes the proof of~\eqref{eq:FirstIterateTori} in Theorem~\ref{th:main}.
\end{proof}

%\bibliography{biblio}{}
%\bibliographystyle{abbrv}

\end{document}